\documentclass[a4paper, 10pt, american]{amsart}
\usepackage{geometry}
\usepackage{times,latexsym,amssymb}
\usepackage{amsmath,amsthm,bm}
\usepackage{color}
\usepackage[colorlinks,pdfpagelabels,pdfstartview = FitH,bookmarksopen
= true,bookmarksnumbered = true,linkcolor = blue,plainpages =
false,hypertexnames = false,citecolor = red,pagebackref=false]{hyperref}
\usepackage{soul}
\usepackage{microtype} 
\usepackage{amsbsy}
\usepackage{amstext}
\usepackage{amssymb}
\usepackage{esint}
\usepackage{stmaryrd}
\usepackage[pdftex]{graphicx}
\usepackage{floatflt}
\usepackage{comment}
\usepackage{cancel}

\allowdisplaybreaks
\newtheorem{mytheorem}{Theorem}
\newtheorem{mylem}[mytheorem]{Lemma}
\newtheorem{mydef}[mytheorem]{Definition}

\newtheorem{remark}[mytheorem]{Remark}
\numberwithin{mytheorem}{section}
\numberwithin{equation}{section}

\usepackage{stmaryrd}
\usepackage{bm}
\SetSymbolFont{stmry}{bold}{U}{stmry}{m}{n}
\usepackage{schemata}
\usepackage{tikz}
\usetikzlibrary{decorations.pathreplacing,calc,tikzmark}
\usetikzlibrary{matrix,decorations.pathreplacing}
\usepackage{blindtext}
\newcommand{\uproman}[1]{\uppercase\expandafter{\romannumeral#1}}
\usepackage[utf8]{inputenc}
\usepackage{mathtools}
\usepackage{amsthm}

\DeclareMathOperator\supp{spt}
\DeclareMathOperator\divv{div}
\DeclareMathOperator\diam{diam}

\usepackage{amsmath}
\usepackage{mathrsfs}
\usepackage{yfonts}
\usepackage{braket}
\usepackage{bbm}
\usepackage{color}
\usepackage{paralist}
\usepackage{pifont}
\usepackage{enumerate}
\usepackage{inputenc}
\usepackage{graphicx}
\newcommand{\bigchi}{\scalebox{1.3}{$\chi$}}

\def\Xint#1{\mathchoice
{\XXint\displaystyle\textstyle{#1}}
{\XXint\textstyle\scriptstyle{#1}}
{\XXint\scriptstyle\scriptscriptstyle{#1}}
{\XXint\scriptscriptstyle\scriptscriptstyle{#1}}
\!\int}
\def\XXint#1#2#3{{\setbox0=\hbox{$#1{#2#3}{\int}$ }
\vcenter{\hbox{$#2#3$ }}\kern-.54\wd0}}

\usepackage{times,latexsym,amssymb}
\usepackage{amsmath,amsthm,bm}
\usepackage{color}
\usepackage[colorlinks,pdfpagelabels,pdfstartview = FitH,bookmarksopen
= true,bookmarksnumbered = true,linkcolor = blue,plainpages =
false,hypertexnames = false,citecolor = red,pagebackref=false]{hyperref}
\makeatletter
\DeclareRobustCommand*{\bfseries}{%
  \not@math@alphabet\bfseries\mathbf
  \fontseries\bfdefault\selectfont
  \boldmath
}
\makeatother
\usepackage{amsbsy}
\usepackage{amstext}
\usepackage{amssymb}
\usepackage{esint}
\usepackage{stmaryrd}
\usepackage[pdftex]{graphicx}
\usepackage{floatflt}
\usepackage{comment}

\newcommand{\foo}[1]{\mathbf{#1}}
\allowdisplaybreaks
\newtheorem{myproposition}[mytheorem]{Proposition}
\numberwithin{mytheorem}{section}
\numberwithin{equation}{section}

\hypersetup{linkcolor=blue,linktoc=page}
\usepackage{titletoc}
\usepackage{schemata}
\usepackage{tikz}
\usetikzlibrary{decorations.pathreplacing,calc,tikzmark}
\usetikzlibrary{matrix,decorations.pathreplacing}
\usepackage{blindtext}
\usepackage[utf8]{inputenc}
\usepackage{mathtools}
\usepackage{amsthm}
\def\Yint#1{\mathchoice
    {\YYint\displaystyle\textstyle{#1}}%
    {\YYint\textstyle\scriptstyle{#1}}%
    {\YYint\scriptstyle\scriptscriptstyle{#1}}%
    {\YYint\scriptscriptstyle\scriptscriptstyle{#1}}%
      \!\iint}
\def\YYint#1#2#3{{\setbox0=\hbox{$#1{#2#3}{\iint}$}
    \vcenter{\hbox{$#2#3$}}\kern-.51\wd0}}
\def\longdash{{-}\mkern-3.5mu{-}} 
\def\fiint{\Yint\longdash}
\def\Xint#1{\mathchoice
{\XXint\displaystyle\textstyle{#1}}%
{\XXint\textstyle\scriptstyle{#1}}%
{\XXint\scriptstyle\scriptscriptstyle{#1}}%
{\XXint\scriptscriptstyle\scriptscriptstyle{#1}}%
\!\int}
\def\XXint#1#2#3{{\setbox0=\hbox{$#1{#2#3}{\int}$ }
\vcenter{\hbox{$#2#3$ }}\kern-0.555\wd0}}

\def\fint{\Xint-}

\DeclareMathOperator\dist{dist}

\usepackage{amsmath}
\usepackage{mathrsfs}
\usepackage{yfonts}
\usepackage{braket}
\usepackage{bbm}
\usepackage{color}
\usepackage{paralist}
\usepackage{pifont}
\usepackage{enumerate}
\usepackage{inputenc}
\usepackage{graphicx}

\renewcommand{\d}{\mathrm{d}}
\newcommand{\dx}{\mathrm{d}x}
\newcommand{\dr}{\mathrm{d}r}

\newcommand{\dt}{\mathrm{d}t}
\newcommand{\ds}{\mathrm{d}s}
\newcommand{\dtau}{\mathrm{d}\tau}

\newcommand{\R}{\mathbb{R}}
\newcommand{\N}{\mathbb{N}}
\renewcommand{\epsilon}{\varepsilon}
\newcommand{\critical}{\frac{2n}{n+2}}
\renewcommand{\d}{\mathrm{d}}

\newcommand{\G}{\mathcal{G}}

\newcommand{\Q}{\mathcal{Q}}

\renewcommand{\epsilon}{\varepsilon}

\begin{document}

\title[Boundary regularity for parabolic systems]{Boundary regularity for parabolic systems with nonstandard \lowercase{\texorpdfstring{$(p,q)$}{}}-growth conditions in smooth convex domains}
\date{\today}


\author[M. Strunk]{Michael Strunk}
\address{Michael Strunk \\
Fachbereich Mathematik, Universit\"at Salzburg \\
Hellbrunner Str. 34, 5020 Salzburg, Austria}
\email{michael.strunk@plus.ac.at}


\begin{abstract}
We study the boundary regularity of local weak solutions to nonlinear parabolic systems of the form
\begin{equation*}
    \partial_t u^i - \mathrm{div} \big( a(|Du|) Du^i \big)= f^i, \qquad i=1,\dots,N,
\end{equation*}
in a space-time cylinder $\Omega_T = \Omega \times (0,T)$, where $\Omega \subset \mathbb{R}^n$ ($n \ge 2$) is a bounded, convex $C^2$-domain and $T>0$. The inhomogeneity $f=(f^1,\dots,f^N)$ belongs to $L^{n+2+\sigma}(\Omega_T,\R^N)$ for some $\sigma>0$. The coefficients $a\colon \R_{>0} \to \R_{>0}$ are of Uhlenbeck-type and satisfy a nonstandard $(p,q)$-growth condition with
\[ 
2 \leq p \leq q < p + \frac{4}{n+2}. 
\] 
Our main result establishes a local Lipschitz estimate up to the lateral boundary for any local weak solution that vanishes on the lateral boundary of the cylinder.
\end{abstract}


\maketitle
\vspace{-0.5cm}
\tableofcontents


\section{Introduction and statement of the main result} \label{sec:introduction}

In this manuscript, we study the boundary regularity of local weak solutions to parabolic systems of the type
\begin{equation} \label{pde}
   \partial_t u^i - \divv \big(a(|Du|) Du^i \big)= f^i, \qquad i=1,\dots,N,
\end{equation}
in a space-time cylinder $\Omega_T \coloneqq \Omega \times (0,T) \subset \R^{n+1}$, 
where $\Omega \subset \R^n$ ($n \ge 2$) is a bounded, open, convex $C^2$-set and $T>0$. 
The coefficients $a\colon \R_{>0} \to \R_{>0}$ are of class $C^1$ and satisfy nonstandard $(p,q)$-growth conditions with
\[
2 \le p \le q < p + \frac{4}{n+2}.
\]
A precise description of the assumptions on the coefficients $a$ is deferred to Section~\ref{subsec:structurecond}. 
For the right-hand side $f=(f^1,\dots,f^N)\colon \Omega_T \to \R^N$, we only assume
\[
f \in L^{n+2+\sigma}(\Omega_T,\R^N)
\]
for some $\sigma>0$, so that no weak differentiability of $f$ is required. As the main result of this article (Theorem~\ref{hauptresultat}), we establish a local~$L^\infty$-gradient estimate for local weak solutions $u$ that holds up to the lateral boundary, provided that $u$ vanishes on the respective part of the boundary in the sense of traces. More precisely, we show that
\[
Du \in L^\infty(\Omega_T \cap Q_\rho(z_0), \R^{Nn}),
\]
for any cylinder 
\[
Q_\rho(z_0) \coloneqq B_\rho(x_0) \times \Lambda_\rho(t_0) 
\coloneqq B_\rho(x_0) \times (t_0-\rho^2, t_0],
\] 
with $u \equiv 0$ on $(\partial \Omega)_T \cap Q_{2\rho}(z_0)$, assuming $\Lambda_{2\rho}(t_0) \subset (0,T)$ and $x_0 \in \partial \Omega$ is a lateral boundary point. In addition, we provide a precise local quantitative $L^\infty$-gradient estimate, stated in~\eqref{est:hauptresultat} of Theorem~\ref{hauptresultat}. This estimate is consistent with previous boundary results in the elliptic $(p,q)$-setting~\cite[Theorem~1.1]{boundaryellipticpq} and the parabolic $(p,p)$-case~\cite[Theorem~1.2]{boundaryparabolic}. Moreover, Theorem~\ref{hauptresultat} extends local interior results for nonstandard $(p,q)$-growth equations, such as~\cite[Theorem~5.1]{bogelein2013parabolic} and~\cite[Theorem~1.3]{singer2015parabolic} in the homogeneous case $f\equiv 0$, to the boundary. 


\subsection{Literature overview} \label{subsec:literature}
The first interior $C^{1,\alpha}$-regularity result for weak solutions to elliptic equations 
(i.e. in the stationary case $N=1$) was obtained by Ural'ceva~\cite{ural1968degenerate}, 
who treated $p$-Laplacian type equations in the super-quadratic range $p\ge 2$. 
Uhlenbeck~\cite{uhlenbeck1977regularity} generalized this result to elliptic $p$-Laplacian type 
systems ($N\ge 1$) for $p\ge 2$. Further generalizations were obtained by Giaquinta and Modica~\cite{giaquinta1986partial,giaquinta1986remarks}. 
The singular sub-quadratic case $1<p<2$ was subsequently addressed by Tolksdorf~\cite{tolksdorf1984regularity} 
and by Acerbi and Fusco~\cite{acerbi1989regularity}, thereby completing the local $C^{1,\alpha}$-regularity theory for $p$-Laplacian type systems for the entire range $p>1$. 

For some time it remained unclear whether the same regularity could be expected for the evolutionary counterpart, that is, for parabolic $p$-Laplacian systems. 
A landmark result of DiBenedetto and Friedman~\cite{friedman1984regularity,dibenedetto1985addendum,Friedman1985}, 
see also~\cite{dibenedetto1993degenerate}, established that a similar interior regularity indeed holds, 
in the sense of local Hölder continuity of the spatial gradient. A critical exponent $\critical$ arises for parabolic $p$-Laplacian systems, 
representing a lower threshold for which weak solutions may become unbounded in the subcritical range $1<p\le \critical$. 
In fact, weak solutions remain locally bounded in this regime provided they satisfy a suitable higher integrability condition~\cite{dibenedetto1993degenerate}.  Consequently, the local Hölder continuity of the spatial gradient applies only to solutions satisfying this condition.  The proof of DiBenedetto and Friedman relies crucially on the technique of intrinsic scaling, 
which captures the natural geometry of the system.

The situation is different for systems with nonstandard $(p,q)$-growth, where $p\le q$. 
In the elliptic setting, the local Lipschitz regularity of weak solutions has been studied extensively, cf.~\cite{boundaryellipticpq,carozza2011higher,carozza2013regularity,colombo2015bounded,
eleuteri2016lipschitz,eleuteri2020regularity,marcellinieins,marcellinizwei,marcellinidrei,
marcellinivier,marcellini2006nonlinear,schmidt2008regularity}. 
By contrast, in the parabolic case, only a few results are known~\cite{bogelein2013parabolic,jiangsheng1998regularity,singer2015parabolic}. 
In~\cite{bogelein2013parabolic,singer2015parabolic}, the authors established quantitative local $L^\infty$-gradient estimates, 
i.e., local Lipschitz regularity, for scalar equations ($N=1$) under slightly different $(p,q)$-growth conditions, 
allowing also for an $x$-dependent vector field. 
As discussed in Section~\ref{subsec:structurecond}, the parameter ranges considered in~\cite[Theorem~1.6]{bogelein2013parabolic} 
and~\cite[Theorem~4.1]{singer2015parabolic} coincide with our range~\eqref{parameter}. 
Finally, in~\cite{jiangsheng1998regularity}, the local $C^{1,\alpha}$-regularity of parabolic systems 
with more general nonstandard growth conditions was obtained. 

Concerning global regularity, i.e., regularity up to the boundary of the domain, the picture is less complete. 
In the elliptic setting, Hamburger~\cite{hamburger1992regularity} established $C^{1,\alpha}$-regularity 
of weak solutions to $p$-Laplacian type systems up to the boundary $\partial\Omega\subset\R^n$, 
provided the solution satisfies either homogeneous Dirichlet or Neumann conditions and the boundary 
is of class $C^{1,1}$. Whether global $C^{1,\alpha}$-regularity persists for general Dirichlet or Neumann data remains an open problem. Foss~\cite{foss2008global} proved that minimizers of  integral functionals with asymptotic $p$-growth 
are Lipschitz continuous up to the boundary if both $\partial\Omega$ and the boundary datum $g$ are of class $C^{1,\alpha}$ for some $\alpha\in(0,1]$. 
Under homogeneous Dirichlet or Neumann conditions, Cianchi and Maz'ya~\cite{cianchi2010global,cianchi2014global} 
provided global Lipschitz estimates for solutions to scalar $p$-Laplacian equations ($N=1$) 
that depend on the regularity of the right-hand side $f$ and hold for a general class of domains, 
including bounded convex sets. Similarly, Lieberman~\cite{lieberman1988boundary} established global 
$C^{1,\alpha}$-regularity for $p$-Laplacian type equations under Dirichlet and conormal Neumann boundary conditions, 
assuming that both the boundary and the boundary data are of class $C^{1,\alpha}$; see also~\cite{antonini2026local} for a recent extension to Orlicz-type equations. In contrast, Banerjee and Lewis~\cite{banerjee2014gradient} proved a result of purely local character. 
Assuming homogeneous Neumann data only on a local part of the boundary, they established a local 
$L^\infty$-gradient estimate in terms of the $L^p$-norm of the gradient in a small interior neighborhood 
of the boundary point. Their analysis concerns the elliptic $p$-Laplace system in bounded convex domains. More recently, Bögelein, Duzaar, Marcellini, and Scheven~\cite{boundaryellipticpq} considered elliptic systems 
with nonstandard $(p,q)$-growth. Their result is local in nature and provides a local $L^\infty$-gradient estimate up to the boundary under homogeneous Dirichlet conditions in bounded and convex sets. As in~\cite{banerjee2014gradient}, 
the estimate depends on the solution in a small neighborhood of the boundary point and on the right-hand side $f$, 
assumed to belong to $L^{n+\sigma}(\Omega,\R^N)$ for some $\sigma>0$. In terms of Lebesgue spaces, this condition 
cannot generally be weakened to $f\in L^n(\Omega,\R^N)$. The assumption $f\in L^{n+\sigma}(\Omega,\R^N)$, 
however, is not sharp; the optimal regularity is $f\in L(n,1)(\Omega,\R^N)$, cf.~\cite{cianchi2010global,cianchi2014global,duzaar2010local}. 
Indeed, a comparable local boundary gradient estimate under homogeneous Dirichlet conditions and 
$f \in L(n,1)(\Omega,\R^N)$ holds for non-uniformly elliptic $(p,q)$-systems as shown by De Filippis and Piccinini~\cite{cristiana}. 

Regarding parabolic $p$-Laplacian type systems, i.e., the case $p=q$ in our structural assumptions~\eqref{parameter}, fewer results are available compared to the elliptic regime. 
Chen and DiBenedetto~\cite{dibenedetto1989boundary} considered parabolic $p$-Laplacian systems with Uhlenbeck-type structure in $C^{1,\alpha}$-domains for $\alpha\in(0,1)$, assuming sufficiently smooth Dirichlet data on the lateral boundary $\partial\Omega\times(0,T)$. 
They first established almost Lipschitz regularity of weak solutions, i.e., $u\in C^{0,\beta}$ for any $\beta\in(0,1)$, and, under homogeneous lateral boundary conditions, obtained global Hölder continuity of the spatial gradient. In~\cite{dibenedetto1992note}, DiBenedetto, Manfredi, and Vespri proved Lipschitz regularity for parabolic $p$-Laplacian equations up to the lateral boundary, under the assumption that the boundary datum is $C^2$ in space and $C^1$ in time, and that the domain is smooth. 
Lieberman~\cite{lieberman1990boundary} further established Hölder continuity of the spatial gradient up to the lateral boundary for weak solutions in $C^{1,\beta}$-domains with $\beta\in(0,1]$, assuming conormal boundary values. A novel contribution in the parabolic context was made by Bögelein~\cite{bogelein2015global}, who proved Lipschitz regularity up to the boundary for a general class of asymptotically parabolic $p$-Laplacian systems. 
Her estimates are local and depend only on the $L^p$-norm of the gradient of the solution and on the boundary data. Specifically, she assumed $\partial\Omega\in C^{1,\alpha}$ and $Dg\in C^{0,\alpha}$ for some $\alpha\in(0,1)$, where $g$ denotes the boundary datum, which is further assumed to possess a distributional derivative in a suitable Morrey space. Finally, we mention the result by Bögelein, Duzaar, Liao, and Scheven~\cite{boundaryparabolic}, who treated parabolic $p$-Laplacian systems with Uhlenbeck-type structure in bounded convex domains. 
Their result can be seen as a parabolic counterpart to the elliptic boundary result in~\cite{boundaryellipticpq} for standard $p$-growth. 
They proved a local $L^\infty$-gradient bound on a neighborhood of the lateral boundary $\partial\Omega\times(0,T)$, assuming that the solution takes homogeneous Dirichlet values there. 
This local boundary gradient estimate provides motivation for the present work and serves as a starting point for our investigation of parabolic systems with nonstandard $(p,q)$-growth. 


\subsection{Structure conditions} \label{subsec:structurecond}

We impose the following structure conditions on the coefficients~$a\colon\R_{> 0} \to \R_{> 0}$. Firstly,~$a$ is assumed to be of class~$C^1(\R_{>0},\R_{> 0})$ and there holds the vanishing condition
\begin{equation} \label{A-Voraussetzung1}
    \lim\limits_{s\downarrow 0} a(s)s = 0. 
\end{equation} 
Moreover, the following nonstandard~$(p,q)$-growth conditions are imposed
\begin{equation} \label{A-Voraussetzung2}
    C_1 (\mu^2+s^2)^{\frac{p-2}{2}} \leq a(s) + a'(s)s \leq C_2  (\mu^2+s^2)^{\frac{q-2}{2}} \qquad \mbox{for any~$s>0$}
\end{equation}
for a parameter~$\mu\in(0,1]$ and some positive constants~$0<C_1\leq C_2$. Here, the parameters~$p,q$ are subject to the relation
\begin{equation} \label{parameter}
    2 \leq p \leq q < p + \frac{4}{n+2}. 
\end{equation}
This relation arises in view of the elliptic counterpart treated in~\cite{boundaryellipticpq}, where the respective assumed relation of parameters is given by~$1<p\leq q< p\big(1+\frac{2}{n}\big)$. To illustrate this matter of fact, we rewrite
$$ p+\frac{4}{n+2} = p+\frac{2p}{n+2} \cdot \frac{2}{p}, $$ 
where the space-dimension~$n$ has to be replaced by~$n+2$ in the parabolic setting due to the geometry of the considered parabolic time intervals, and the additional factor~$\frac{2}{p}$ denotes the usual parabolic deficit arising from the inhomogeneous behavior of the system~\eqref{pde}.
Moreover, this range of parameters also coincides with the one given in~\cite{bogelein2013parabolic,singer2015parabolic}. We note that the assumption~$\mu\in (0,1]$ implies that only non-degenerate systems are considered in the course of this article. 


\begin{remark} \upshape 
    The assumption~\eqref{A-Voraussetzung2} has to be understood in the sense that the structure parameters~$\mu,C_1,C_2$ are chosen appropriately in such a way that~\eqref{A-Voraussetzung2} is satisfied and the relation
    $$ C_1 \mu^{p-q} \leq C_2 $$
    holds true, that is a necessary condition which becomes evident when~$s\downarrow 0$ in~\eqref{A-Voraussetzung2}. Indeed, this additional assumption on the constants~$C_1,C_2$ is negligible by replacing the non-degeneracy quantity~$\mu^2$ with~$1$ in~\eqref{A-Voraussetzung2}. We further note that the missing degenerate case where~$\mu=0$ in structure condition~\eqref{A-Voraussetzung2} can only be treated if the coefficients~$a$ satisfy standard~$p$-growth, which prevails only if~$p=q$ holds true. In general however, only the nondegenerate case~$\mu\in(0,1]$ is considered.  
\end{remark} 


For the inhomogeneity~$f=(f^1,\ldots,f^N)$ of the system~\eqref{pde} we assume that
\begin{equation} \label{datumregularität}
    f\in L^{n+2+\sigma}(\Omega_T,\R^N)
\end{equation}
for some parameter~$\sigma>0$.


\begin{remark} \upshape
Based on the assumption~\eqref{parameter}, our analysis is focused on the super-quadratic case where~$p \geq 2$. Nonetheless, a similar approach could be employed to address the sub-quadratic super-critical regime, where~$\critical < p < 2$. For this, the nonstandard~$(p,q)$-growth conditions specified in~\eqref{A-Voraussetzung2} would need to be modified appropriately.
\end{remark}


\subsection{Main result} \label{subsec:mainresult}

We hereby state our main regularity result, which establishes a local~$L^\infty$-gradient estimate for any local weak solution to~\eqref{pde}, which holds true up to the lateral boundary of the domain, provided the solution admits local homogeneous boundary values. 

\begin{mytheorem} \label{hauptresultat}
    Let~$\Omega\subset\R^n$ denote a bounded and convex~$C^2$-set and let
    $$u\in C^0\big(\Lambda_{2\rho}(t_0);L^2(\Omega\cap B_{2\rho}(x_0),\R^N)\big) \cap L^q\big(\Lambda_{2\rho}(t_0);W^{1,q}(\Omega\cap B_{2\rho}(x_0),\R^N) \big)$$
    be a local weak solution to the parabolic system~\eqref{pde} under the set of assumptions~\eqref{A-Voraussetzung1} -- \eqref{datumregularität} that satisfies the boundary condition
    $$u\equiv 0\mbox{\qquad on~$(\partial\Omega)_T \cap Q_{2\rho}(z_0)$},$$
    where~$\Lambda_{2\rho}(t_0)\subset(0,T)$. Then, for any parabolic cylinder~$Q_{\rho}(z_0)=B_{\rho}(x_0)\times\Lambda_{\rho}(t_0)$ with~$x_0\in \partial\Omega$ there holds
    $$Du\in L^\infty(\Omega_T \cap Q_\rho(z_0),\R^{Nn}). $$
    Moreover, we have the quantitative local~$L^\infty$-gradient estimate 
    \begin{align} \label{est:hauptresultat}
        & \sup\limits_{\Omega_T\cap Q_{\rho}(z_0)}|Du| \\
        &\qquad\,\,\leq C \bigg( \Big( 1 + \|f\|^{\frac{(n+2)(n+2+\sigma)}{\sigma}}_{L^{n+2+\sigma}(\Omega_T\cap Q_{2\rho}(z_0))} \rho^{n+2} \Big)  \fiint_{\Omega_T \cap Q_{2\rho}(z_0)} (1+|Du|^p) \,\dx\dt \bigg) ^{\frac{1}{p} \cdot \frac{2p}{4+(p-q)(n+2)}} \nonumber
    \end{align}
    with~$C=C\big(C_1,C_2,n,p,q,\sigma,\Theta_{\rho}(x_0) \big)$.
\end{mytheorem}


  The boundary values~$u\equiv 0$ on~$(\partial\Omega)_T \cap Q_{2\rho}(z_0)$ are to be understood in the sense of traces. The geometric quantity~$\Theta_\rho(x_0)$ depends on the bounded and convex $C^2$-set~$\Omega$ and is explained in more detail in Section~\ref{subsec:convex}. 


\begin{remark} \label{remarknachhauptresultatzwei} \upshape 
The factor
$$ \frac{2p}{4+(p-q)(n+2)} $$ 
in the exponent present on the right-hand side of~\eqref{est:hauptresultat} denotes the scaling deficit of the parabolic system~\eqref{pde} and leads to an inhomogeneous nature of estimate~\eqref{est:hauptresultat}. Firstly, we note that our scaling deficit coincides with the one from~\cite[Theorem~5.1]{bogelein2013parabolic} in the super-quadratic case~$p\geq 2$ and also with~\cite[Theorem~4.1]{singer2015parabolic}. Moreover, in case of standard~$p$-growth, i.e. when~$p=q$, we have
$$ \frac{2p}{4+(p-q)(n+2)} = \frac{2p}{p(n+2)-2n} $$
and thus recover the scaling deficit for parabolic systems with~$p$-growth from~\cite[Theorem~1.2]{boundaryparabolic}. We further remark that the scaling deficit would blow up when~$p$ approaches the critical value~$\critical$, i.e. when~$p\downarrow \critical$, which does not arise in the treatment of our super-quadratic case~$p\geq 2$ throughout this manuscript. Similarly, our scaling deficit blows up when~$q$ approaches the upper bound imposed in~\eqref{parameter}, i.e. when~$q\uparrow (p+\frac{4}{n+2})$. 
\end{remark} 


\begin{remark} \label{remarknachhauptresultatdrei} \upshape

We note that a weaker estimate in comparison to~\eqref{est:hauptresultat} is obtained in Section~\ref{subsec:moser}, estimate~\eqref{est:iterationfünf}, which involves the stronger~$L^q$-norm of~$Du$ on the right-hand side. Thereby, the appearing scaling deficit is given by
$$ \frac{1}{q} \cdot \frac{2q}{4+n(p-q)}, $$
which again equals the respective scaling deficits in~\cite{boundaryparabolic,bogelein2013parabolic,singer2015parabolic}. Indeed, the weaker assumption
$$ 2\leq p\leq q<p+\frac{4}{n} $$
turns out sufficient at this point. The reason for this is explained pretty easily: in the parabolic setting, Sobolev's embedding improves the integrability from~$p$ to~$p+\frac{4}{n}$. Our approach in order to establish~\eqref{est:iterationfünf} builds upon a Moser iteration procedure, for which Caccioppoli-type estimates and Sobolev's embedding are the major components. Together, they lead to a reverse Hölder-type inequality, only if the prior assumptions made on~$p,q$ pertains. \,\\
For the approximation procedure in Section~\ref{sec:approx} however, it is essential to control the~$L^\infty$-gradient norm by its~$L^p$-norm rather than the~$L^q$-norm. This is achieved by exploiting a weaker estimate involving the~$L^q$-gradient norm on the right-hand side and applying an interpolation argument. Thereby, the assumptions on the parameters~$p,q$ in~\eqref{parameter}, as well as the~\textit{a priori} boundedness of the gradient are crucial. 
\end{remark}


\begin{remark} \label{remarknachhauptresultatvier} \upshape

According to~\eqref{pde}, the coefficients~$a(|Du|)$ are radially symmetric with respect to~$Du\in\R^{Nn}$. Such a structure is usually referred to as an~\textit{Uhlenbeck-type structure}. Due to this specific structure of the parabolic system~\eqref{pde}, one can~\textit{a priori} expect the claimed local gradient bound~\eqref{est:hauptresultat} at least heuristically. However, in the case of more general structures, the best one can generally hope for is partial regularity and weak solutions may exhibit singularities within the domain. For a detailed explanation of this phenomenon, we refer to \cite{giaquinta1986partial, giusti2003direct, mingionepartial}.

\end{remark}


\subsection{Strategy of proof} \label{subsec:strategy}

In the following, we provide a few words on the proof of the main result that is Theorem~\ref{hauptresultat}. In Section~\ref{sec:aprioriest}, we commence by proving a quantitative~\textit{a priori}~$L^\infty$-gradient estimate for smooth classical solutions~$u\in C^3(\overline{\Omega}_T\cap Q_{2\rho}(z_0),\R^N)$ to a more regular version of the parabolic system~\ref{pde}, that is given by
\begin{equation} \label{approxsystem}
    \partial_t u^i - \divv \big(a(|Du|) Du^i \big) = f^i \qquad\mbox{in~$\Omega_T\cap Q_{2\rho}(z_0)$, \quad for~$i=1,\ldots,N$},
\end{equation}
where~$\Lambda_{2\rho}(t_0)\subset(0,T)$ and for the datum~$f$ we assume that~$f\in C^1(\overline{\Omega}_T\cap Q_{2\rho}(z_0),\R^N)$ with~$\supp f \Subset\Omega\times\R$. As in the statement of the main regularity result, Theorem~\ref{hauptresultat}, the solutions~$u\in C^3(\overline{\Omega}_T\cap Q_{2\rho}(z_0),\R^N)$ are prescribed to take homogeneous lateral boundary values~$u\equiv 0$ on~$(\partial\Omega)_T\cap Q_{2\rho}(z_0)$. The reason for~$u$ to belong to~$C^3(\overline{\Omega}_T\cap Q_{2\rho}(z_0),\R^N)$ instead of the assumption~$u\in C^2(\overline{\Omega}_T\cap Q_{2\rho}(z_0),\R^N)$ for classical solutions is due to the fact that we aim to establish a Caccioppoli-type inequality through a differentiation of the classical form of the parabolic system~\eqref{approxsystem}, rather than it's weak formulation, necessitating the existence of third order classical derivatives. Moreover, the~$C^2$-regularity of the bounded and convex set~$\Omega$, together with the assumption of homogeneous lateral boundary values for~$u$, is crucial in order to obtain the Caccioppoli-type estimate~\eqref{est:eesecondorder} in Lemma~\ref{lem:eesecondorder}. These conditions allow the application of a differential geometric identity that is taken from Grisvard~\cite{grisvard2011elliptic}, cf.~\cite[Lemma~2.1]{boundaryparabolic}. Roughly speaking, it allows to discard a boundary integral that exhibits a sign, where we refer the reader to~\cite[Proof of Proposition~3.2]{boundaryellipticpq},~\cite[Section~4.1]{cristiana}, and~\cite[Proof of Proposition~3.2]{boundaryparabolic} for the details. We merely remark that the assumption of homogeneous Dirichlet boundary values for~$u$ imply that its tangential derivative vanishes, which is a crucial ingredient for the application of the differential geometric identity from~\cite{grisvard2011elliptic}. Hereinafter, we exploit the Caccioppoli-type estimate from Lemma~\ref{lem:eesecondorder} and establish a reverse Hölder-type inequality in form of Lemma~\ref{lem:reverseholder}. For this, the assumption~\eqref{parameter} on the growth parameters~$p,q$ could actually be weakened to
$$2\leq p\leq q<p+\frac{4}{n},$$
which is essentially due to the fact that Sobolev's embedding in the parabolic setting improves the integrability from~$p$ to~$p+\frac{4}{n}$. This reverse Hölder-type estimate is the starting point for a Moser iteration procedure. This is performed in Section~\ref{subsec:moser}, which overall results in the \textit{a priori}~$L^\infty$-gradient estimate~\eqref{est:apriorigradientsuper} in Proposition~\ref{prop:apriorigradientsuper}. We further remark that in the $L^\infty$-gradient estimate the $q$-energy arises, an immediate consequence of the growth of the parabolic system in view of assumption \eqref{A-Voraussetzung2}. However, this estimate can be interpolated using a classical iteration Lemma \ref{lem:iteration}, combined with another exploitation of the parameter relation \eqref{parameter}, which allows to control the $L^\infty$-gradient norm of $u$ through its $p$-energy. Moreover, the very scaling deficit of the parabolic system~\eqref{pde} from Remark \ref{remarknachhauptresultatzwei} arises, preventing the estimate from being homogeneous in general. 

Section~\ref{sec:approx} sets up an approximation procedure of the original parabolic system~\ref{pde} in order to exploit the previously derived quantitative~$L^\infty$-gradient estimate for smooth classical solutions~$C^3(\overline{\Omega}_T\cap Q_{2\rho}(z_0),\R^N)$ in the setting from above. We mollify the coefficients~$a\in C^1(\R_{>0},\R_{>0})$ and also the datum~$f\in L^{n+2+\sigma}(\Omega_T,\R^N)$ to obtain smooth coefficients~$a_\epsilon$ and~$f_\epsilon\in C^\infty$ with~$\supp f\Subset \Omega_\epsilon\times\R$. It then follows that~$a_\epsilon$ again satisfies comparable nonstandard~$(p,q)$-growth conditions, as outlined in~\cite[Appendix~A]{boundaryparabolic}, and there holds~$a_\epsilon(s)\to a(s)$ pointwise for any~$s>0$, as stated in~\eqref{smallnessest}. We then define the actual approximation coefficients~$b_\epsilon$ by incorporating a non-degenerate~$q$-Laplacian-like term into~$a_\epsilon$, i.e. we consider
$$ b_\epsilon(s) \coloneqq \epsilon(\mu^2+s^2)^{\frac{q-2}{2}}+a_\epsilon(s) $$
for~$s>0$. Consequently, the coefficients~$b_\epsilon$ satisfy standard~$q$-growth conditions, with structure constants depending on~$\epsilon\in(0,1]$. Next, we consider the unique local weak solution~$u_\epsilon$ to the Cauchy-Dirichlet problem 
\begin{align*} 
    \begin{cases}
        \partial_t u^i_\epsilon - \divv \big( b_{\epsilon}(|Du_\epsilon|) Du^i_\epsilon \big)= f^i_\epsilon  & \mbox{in~$\Omega_T \cap Q_{2\rho}(z_0)$}, \\
       u^i_\epsilon = u^i & \mbox{on~$\partial_p \big(\Omega_T \cap Q_{2\rho}(z_0)\big)$},
    \end{cases}
\end{align*}
for any $i=1,\ldots,N\in\N$ and~$\epsilon\in(0,1]$, where~$x_0\in \partial\Omega$ and~$\Lambda_{2\rho}(t_0)=(t_0-4\rho^2,t_0]\subset(0,T)$. The existence of the solutions~$u_\epsilon$ follows from classical existence theory as a consequence of standard~$q$-growth, cf.~\cite{lions1969quelques}. We then establish a comparison estimate between~$u_\epsilon$ and~$u$, which subsequently implies the strong convergence of the gradients~$Du_\epsilon \to Du$ in~$L^p(\Omega_T\cap Q_{2\rho}(z_0),\R^N)$ as~$\epsilon\downarrow 0$. The remaining task is to verify that the approximating solutions~$u_\epsilon$ are indeed regular, classical solutions. Since the coefficients~$b_\epsilon$ satisfy standard~$q$-growth conditions, we refer to~\cite[Appendix~B]{boundaryparabolic}, where a similar result has already been established. This allows us to conclude that~$u_\epsilon \in C^3(\overline{\Omega}_T\cap Q_{2\rho}(z_0),\R^N)$ for any~$\epsilon\in(0,1]$. Contrary to the settings in~\cite{boundaryparabolic,boundaryellipticpq}, where only bounded and convex sets are considered, our analysis requires the additional assumption that~$\Omega$ is a bounded and convex~$C^2$-set. The necessity for this regularity condition stems from our approximation procedure outlined in Section~\ref{sec:approx}. 

The final Section~\ref{sec:hauptresultatbeweis} is devoted to the proof of Theorem~\ref{hauptresultat}. Thereby, we exploit the quantitative~\textit{a priori}~$L^\infty$-gradient estimate from Proposition~\ref{prop:apriorigradientsuper} and use the comparison estimate from Lemma~\ref{lem:energyestimate}. Moreover, the coefficients~$b_\epsilon$ satisfy similar~$(p,q)$-growth conditions as~$a$, such that Proposition~\ref{prop:apriorigradientsuper} is applicable to the data~$b_\epsilon,f_\epsilon,\Omega$. There, it is important that the constant in the quantitative~\textit{a priori}~$L^\infty$-gradient for~$u_\epsilon$ does not depend on~$\epsilon\in(0,1]$. An application of the comparison estimate from Lemma~\ref{lem:energyestimate} yields that~$Du_\epsilon \to Du$ strongly in~$L^p(\Omega_T\cap Q_{2\rho}(z_0),\R^N)$ as~$\epsilon\downarrow 0$. Passing to a subsequence~$(\epsilon_i)_{i\in\N}$, we further infer that~$Du_{\epsilon_i} \to Du$ pointwise a.e. in~$\Omega_T\cap Q_{2\rho}(z_0)$. Finally, for the right-hand side we have~$f_\epsilon \to f$ in~$L^{n+2+\sigma}(\Omega_T\cap Q_{2\rho}(z_0),\R^N)$ as~$\epsilon\downarrow 0$, such that we may pass to the limit~$\epsilon\downarrow 0$ in the quantitative~$L^\infty$-gradient estimate for~$u_\epsilon$. This results in the main Theorem~\ref{hauptresultat}. 


\subsection{Plan of the paper} \label{subsec:planofpaper}

The paper is organized as follows: in Section~\ref{sec:notationandprelim}, we begin by presenting the notation and framework, including supplementary material required later on. Additionally, we state the notion of a local weak solution to the parabolic system~\eqref{pde}. In Section~\ref{sec:aprioriest}, we give the proof of the quantitative local~\textit{a priori}~$L^\infty$-gradient estimate~\eqref{est:apriorigradientsuper} from Proposition~\ref{prop:apriorigradientsuper}, which holds true for local~$C^3$-solutions with homogeneous Dirichlet boundary values with smooth coefficients and also a smooth right-hand side~$f$. Once this~\textit{a priori} regularity result has been established, we proceed in Section~\ref{sec:approx} by performing an approximation procedure of the original system~\eqref{pde}. Here, we mollify the coefficients~$a$ and also the datum~$f$, such that the~\textit{a priori} gradient estimate from~\ref{prop:apriorigradientsuper} can be applied onto the approximating local solutions~$u_\epsilon$ on the bounded and convex~$C^2$-set~$\Omega$. In Section~\ref{sec:hauptresultatbeweis}, we give the proof of the main regularity theorem that is Theorem~\ref{hauptresultat}. This is achieved by passing to the limit~$\epsilon\downarrow 0$ and employing certain compactness arguments, allowing a transfer of the quantitative~$L^\infty$-gradient estimate from the approximating local solutions~$u_\epsilon$ to the original local weak solution~$u$.  
\,\\


\subsection*{Acknowledgements} The author wishes to extend many thanks to Verena Bögelein and to Frank Duzaar for their support during the development of this article. \\
This research was funded in whole or in part by the Austrian Science Fund (FWF) [10.55776/P36295]. For open access purposes, the author has applied a CC BY public copyright license to any author accepted manuscript version arising from this submission.


\section{Notation and preliminaries} \label{sec:notationandprelim}


\subsection{Notation} \label{subsec:notation}
In the course of this article,~$\Omega_T = \Omega\times(0,T)$ represents a space-time cylinder that is taken over a bounded and open convex set~$\Omega\subset\R^n$, while~$0<T<\infty$ represents a positive and finite time. As usual, the \textit{parabolic boundary} of this space-time cylinder shall be denoted by
$$ \partial_p\Omega_T = \big(\overline{\Omega}\times\{0\}\big) \cup \big(\partial\Omega\times(0,T) \big). $$
The expressions \textit{gradient}, \textit{spatial gradient}, and \textit{spatial derivative} will be used interchangeably to refer to the spatial derivative~$Du$ for any differentiable function~$u$. Accordingly, the \textit{partial derivative} of any differentiable function~$u\colon\Omega_T\to\R^k$, for~$k\in\N$, with respect to~$i\in\{1,...,n\}$ will be denoted either by~$D_i u$ or by~$\partial_{x_i}u$. To distinguish between \textit{spatial} and \textit{time derivatives}, the latter will be denoted as~$\partial_t u$, given that~$u\colon\Omega_T\to\R^k$ admits a derivative with respect to the time variable~$t$. It is worth noting that no distinction between classical and weak derivatives will be made. The positive part of any real-valued quantity~$a\in\R$ is denoted by~$a_+ = \max\{a,0\}$, whereas the negative part is denoted as~$a_- = \max\{-a,0\}$. Constants are consistently represented as~$C$ or~$C(\cdot)$ and their dependence is described solely on their variables, not their specific values. However, it is possible for constants to vary from one line to another without further clarification.
\,\\

We denote the open ball in~$\mathbb{R}^n$ with radius~$\rho>0$ and center~$x_0\in\mathbb{R}^n$ as 
$$B_\rho(x_0) \coloneqq\{x\in\mathbb{R}^n\colon~|x-x_0|<\rho\}.$$
By a point~$z_0\in\R^{n+1}$, we always refer to the vector~$z_0 = (x_0,t_0)$ where~$x_0\in\R^n$ and~$t_0\in\R$. In the case where~$x_0=0$, it will sometimes be convenient to simplify notation by writing~$B_\rho$ instead of~$B_\rho(x_0)$. The standard (backward) parabolic cylinder is given by
$$Q_\rho(z_0)\coloneqq B_\rho(x_0)\times \Lambda_{\rho}(t_0) \coloneqq B_\rho(x_0)\times (t_0-\rho^2,t_0].$$
Let~$A\subset\Omega$ denote an arbitrary set with the property that~$\mathcal{L}^n(A)>0$. Then, we define the \textit{slice-wise mean}~$(f)_A\colon (0,T)\to\R^n$ of~$f$ over~$A$ at time~$t\in(0,T)$ as
\begin{align*} (f)_A(t) \coloneqq \displaystyle\fint_A f(x,t)\,\dx,\quad\mbox{for a.e.~$t\in(0,T)$}.
\end{align*}
The slice-wise mean value is indeed well-defined for all~$t\in(0,T)$ for any function that is of class~$f\in C(0,T;L^2(\Omega))$. Similarly, we denote the \textit{mean value} of~$f$ on some measurable set~$E\subset \Omega_T$ with~$\mathcal{L}^{n+1}(E)>0$ by
\begin{align*}
     (f)_E \coloneqq \displaystyle\fiint_E f(x,t)\,\dx\dt.
\end{align*}


\subsection{Definition of weak solution} \label{subsec:notion}

We hereby state our notion of a local weak solution to the parabolic system~\eqref{pde}. 


\begin{mydef} \label{defweakform}
   Let~$f$ be given as in~\eqref{datumregularität} and the coefficients~$a$ satisfy assumptions~\eqref{A-Voraussetzung1} and~\eqref{A-Voraussetzung2}. A measurable function
   $$ u\in C^0\big(\Lambda_{2\rho}(t_0);L^2(\Omega\cap B_{2\rho}(x_0),\R^N) \big) \cap L^q \big(\Lambda_{2\rho}(t_0);W^{1,q}(\Omega\cap B_{2\rho}(x_0),\R^N) \big) $$
   is a local weak solution to the parabolic system~\eqref{pde} in~$\Omega_T \cap Q_{2\rho}(z_0)$ with~$\Lambda_{2\rho}(t_0)\subset(0,T)$, if the integral identity
\begin{equation} \label{weakform}
    \iint_{\Omega_T \cap Q_{2\rho}(z_0)}  \big( - u \cdot \partial_t \phi + a(|Du|)Du\cdot D\phi \big)  \,\dx\dt =  \iint_{\Omega_T \cap Q_{2\rho}(z_0)} f\cdot \phi\,\dx\dt
\end{equation}
holds true for any test function~$\phi\in C^{\infty}_0(\Omega_T \cap Q_{2\rho}(z_0),\R^N)$. 
\end{mydef} 


\begin{remark} \label{remarknachdefinition} \upshape 
    According to the reasoning presented in Section~\ref{subsec:structurecond}, the coefficients~$a\in C^1(\R_{>0},\R_{>0})$ exhibit a nonstandard~$(p,q)$-growth condition, where the relation between the parameters~$p,q$ is given as in~\eqref{parameter}. Therefore, the existence of a local weak solution to~\eqref{pde} is~\textit{a priori} not entirely obvious. Existence results for equations with~$(p,q)$-growth have been addressed in~\cite{bogelein2013parabolic,singer2015parabolic}, where the Cauchy-Dirichlet problem was studied. This approach can be used in a similar manner in case of  systems with Uhlenbeck-type~\eqref{pde} and with an inhomogeneity~$f$, where only minor adaptions are required in the proof. In particular, a perturbative approach together with a localized Minty-type argument, cf.~\cite{bogelein2011degenerate}, and the established~\textit{a priori} $L^\infty$-gradient estimate from Proposition~\ref{prop:apriorigradientsuper} can be used to prove the existence of a weak solution to~\eqref{pde} satisfying the regularity requirement in Theorem~\ref{hauptresultat} along the lateral boundary. Therefore, the appropriate function space for local weak solutions is the one given in Definition~\ref{defweakform}, i.e.
    $$ C^0\big(\Lambda_{2\rho}(t_0);L^2(\Omega\cap B_{2\rho}(x_0),\R^N) \big) \cap L^q \big(\Lambda_{2\rho}(t_0);W^{1,q}(\Omega\cap B_{2\rho}(x_0),\R^N) \big). $$ 
\end{remark} 


\subsection{On convex domains} \label{subsec:convex}
The quantity
\begin{equation} \label{est:convex}
\Theta_\rho(x_0) = \frac{\rho^n}{|\Omega\cap B_\rho(x_0)|}, 
\end{equation}
which appears in the statement of the main result, that is Theorem~\ref{hauptresultat}, deserves some remark. As already explained in~\cite[Section~2.1]{boundaryparabolic} and also~\cite[Section~2.1]{boundaryellipticpq}, any convex set~$\Omega\subset\R^n$ satisfies a~\textit{uniform cone condition}. As a consequence,~$\Theta_\rho(x_0)$ may be bounded from above by a constant that only depends on the domain~$\Omega$ and that is independent of the considered boundary point~$x_0\in\partial\Omega$ as well as the radius~$\rho>0$. 


\subsection{Auxiliary material} \label{subsec:auxiliary}


\subsubsection{Structure estimates} \label{subsubec:structure}


The growth and monotonicity assumption~\eqref{A-Voraussetzung2} readily yields the following useful lemma.


\begin{mylem} \label{lem:growthnurfürA}
 Let the coefficients~$a$ satisfy assumptions~\eqref{A-Voraussetzung2} with parameter~$\mu\in(0,1]$ and let~$p,q$ satisfy the relation~\eqref{parameter}. Then, there exist positive constants~$c=c(p,C_1), C=C(q,C_2)$, such that
    \begin{equation} \label{est:growthnufürA}
        c (\mu^2+s^2)^{\frac{p-2}{2}} \leq a(s) \leq C (\mu^2+s^2)^{\frac{q-2}{2}}
    \end{equation}
    holds true for any~$s\geq 0$.
\end{mylem}

\begin{proof}
Since the coefficients~$a$ are differentiable on $(0,\infty)$, we calculate with a change of variables
\begin{align*}
    a(s) = \frac{1}{s} \int_{0}^{s} \mbox{$\frac{\d}{\dr}$} (a(r)r) \,\dr = \int_{0}^{1} (a(s\tau) + s\tau a'(s\tau) ) \,\dtau
\end{align*}
for any~$s>0$ and take assumptions~\eqref{A-Voraussetzung1} -- \eqref{A-Voraussetzung2} into account. By estimating~$a$ from below via~\eqref{A-Voraussetzung2} and by employing the result found in~\cite[Lemma~2.1]{acerbi1989regularity}, we obtain
\begin{align*}
    a(s) &= \int_{0}^{1} (a(s\tau) + s\tau a'(s\tau) ) \,\d\tau\geq C_1 \int_{0}^{1} ( \mu^2+ (s\tau)^2)^{\frac{p-2}{2}} \,\d\tau \geq C(p)C_1 (\mu^2+s^2)^{\frac{p-2}{2}},
\end{align*}
which implies the lower bound of the lemma in the case where~$s>0$. The upper found can be inferred in the very same manner.    
\end{proof}


A direct calculation verifies that for any~$\xi\in\R^{Nn}\setminus\{0\}$ there holds
\begin{equation} \label{est:ableitungvonA}
    \partial_{\xi^j_{\beta}} \big[a(|\xi|)\xi^{i}_{\alpha})\big] = a(|\xi|) \delta^{ij}_{\alpha\beta} + \frac{a'(|\xi|)}{|\xi|} \xi^\alpha_i \xi^\beta_j
\end{equation}
for all~$\alpha,\beta\in\{1,\ldots,n\}$ and~$i,j \in \{1,\ldots,N\}$. Therefore, the following positive quantities
    \begin{align} \label{lambda:Lambda}
    \begin{cases}
        h(s) \coloneqq \min\{a(s),a(s) + sa'(s) \}, & \\
        H(s) \coloneqq \max\{a(s),a(s) + sa'(s) \}&
        \end{cases}
    \end{align}
 arise, where~$s > 0$. They will be useful later on. In particular, as an immediate consequence of the growth assumption~\eqref{A-Voraussetzung2}, we obtain the following lemma.


\begin{mylem} \label{lem:lambda:Lambda}
    Let the coefficients~$a$ satisfy assumptions~\eqref{A-Voraussetzung2} with parameter~$\mu\in(0,1]$. For any~$s > 0$ there holds
    \begin{align} \label{est:lambda:Lambda}
    \begin{cases}
         h(s) \geq c(p,C_1) (\mu^2+s^2)^{\frac{p-2}{2}}, & \\
        H(s) \leq C(q,C_2) (\mu^2+s^2)^{\frac{q-2}{2}}.&
    \end{cases}
    \end{align}
\end{mylem}


Lemma~\ref{lem:lambda:Lambda} and~\eqref{est:ableitungvonA} are the starting point for the subsequent lemma, stating a growth and monotonicity property for the vector field in the diffusion term of the parabolic system~\eqref{pde}. It follows similarly to the case~$p=q$, where we refer to~\cite[Section 2.3]{boundaryparabolic}. 


\begin{mylem} \label{lem:monotonicity}
Let the coefficients~$a$ satisfy assumptions~\eqref{A-Voraussetzung2} with parameter~$\mu\in(0,1]$ and let~$\xi\in\R^{Nn}\setminus\{0\}$. Then, there holds
    \begin{equation} \label{est:monotonicity}
        h(|\xi|)|\eta|^2 \leq \sum\limits_{\alpha,\beta=1}^n \sum\limits_{i,j=1}^N \partial_{\xi^j_{\beta}} \big[a(|\xi|)\xi^{i}_{\alpha})\big] \eta^i_\alpha \eta^j_\beta \leq H(|\xi|)|\eta|^2
    \end{equation}
    for any~$\eta\in\R^{Nn}$.
\end{mylem}


Combining the results of Lemma~\ref{lem:lambda:Lambda} and~\ref{lem:monotonicity}, we readily obtain the following ellipticity and growth estimate.


\begin{mylem} \label{lem:ellipticity}
    Let the coefficients~$a$ satisfy assumptions~\eqref{A-Voraussetzung2} with parameter~$\mu\in(0,1]$. Then, there exist positive constants~$\widetilde{C}_1=\widetilde{C}_1(p,C_1), \widetilde{C}_2=\widetilde{C}_2(q,C_2)$, such that for any~$\xi\in\R^{Nn}\setminus\{0\}$ there holds
     \begin{equation} \label{est:ellipticity}
       \widetilde{C}_1(\mu^2+|\xi|^2)^{\frac{p-2}{2}}|\eta|^2 \leq \sum\limits_{\alpha,\beta=1}^n \sum\limits_{i,j=1}^N \partial_{\xi^j_{\beta}} \big[a(|\xi|)\xi^{i}_{\alpha})\big] \eta^i_\alpha \eta^j_\beta \leq \widetilde{C}_2(\mu^2+|\xi|^2)^{\frac{q-2}{2}}|\eta|^2
    \end{equation}
    for any~$\eta\in\R^{Nn}$. 
\end{mylem}


\subsubsection{Algebraic inequalities} \label{subsubsec:algebraic} 

The next algebraic lemma is taken from~\cite[Lemma~2.3]{boundaryparabolic} and will be used in the~\textit{a priori}~$L^\infty$-gradient estimate in Section~\ref{sec:aprioriest}. Thereby, it is used in the course of the Moser iteration procedure in order to estimate the appearing constants. 

\begin{mylem} \label{lem:algebraic}
    Let~$A>1$,~$\kappa>1$,~$\delta>0$, and~$j\in\N$. Then, there holds
    \begin{align} \label{est:algeins}
        \prod_{i=i}^j A^{\frac{\kappa^{j-i+1}}{\delta(\kappa^j-1)}} = A^{\frac{\kappa}{\delta(\kappa-1)}}
    \end{align}
    and also
    \begin{align} \label{est:algzwei}
        \prod_{i=i}^j A^{\frac{i \kappa^{j-i+1}}{\delta(\kappa^j-1)}} \leq A^{\frac{\kappa^2}{\delta(\kappa-1)^2}}.
    \end{align}
\end{mylem}


The subsequent iteration lemma is classical and allows a technique of reabsorbing certain quantities, which stems from~\cite[Lemma~6.1]{giusti2003direct}. 

\begin{mylem} \label{lem:iteration} 
      Let $0< R_0< R_1$, $\phi\colon[R_0,R_1]\to\R$ be a bounded, non-negative function and assume that for $R_0\leq \rho < r \leq R_1$ there holds
      $$\phi(\rho) \leq \eta \phi(r) + \frac{A}{(r-\rho)^{\alpha}} + \frac{B}{(r-\rho)^{\beta}} + C$$
      for some constants $A,B,C,\alpha\geq \beta\geq 0$, and $\eta\in(0,1)$. Then, there exists a constant $\widetilde{C}=\widetilde{C}(\eta,\alpha)$, such that for all $R_0\leq \rho_0<r_0\leq R_1$ there holds
      $$\phi(\rho_0) \leq \widetilde{C}\bigg( \frac{A}{(r_0-\rho_0)^{\alpha}} + \frac{A}{(r_0-\rho_0)^{\alpha}} + C \bigg).$$ 
\end{mylem} 

We recall the following interpolation inequality, which is a consequence of Hölder's inequality. 


\begin{mylem} \label{lem:interpolation} 
    Let~$\Omega\subset\R^n$,~$1\leq p\leq q\leq r$, and~$u\in L^p(\Omega)\cap L^r(\Omega)$. Then, there holds~$u\in L^q(\Omega)$, and for~$\theta\in[0,1]$ with~$\frac{1}{q} = \frac{\theta}{p} + \frac{1-\theta}{r}$, we have the quantitative estimate
    \begin{align*}
        \|u\|_{L^q(\Omega)} \leq \|u\|^{\theta}_{L^p(\Omega)} \|u\|^{1-\theta}_{L^r(\Omega)}.
    \end{align*}
\end{mylem}


\subsubsection{Other material} \label{subsubsec:othermaterial}


The subsequent lemma states a version of Kato's inequality. 

\begin{mylem} \label{lem:kato}
    Let~$k\in\N$ and~$B_R(x_0)\subset\R^n$. For any~$u\in W^{2,1}(B_R(x_0),\R^k)$ there holds
   \begin{equation*}
        |D|Du||\leq |D^2 u|\quad\text{a.e. in $B_R(x_0)$}.
    \end{equation*}
\end{mylem}


The following result is extracted from~\cite[Chapter I, Proposition 3.1]{dibenedetto1993degenerate}.

\begin{mylem} \label{lem:parabolicsobolevembedding}
    Let~$\Omega\subset\R^n$ be a bounded domain and~$\Omega_T = \Omega\times(0,T)$. For any
    $$ u\in L^\infty(0,T;L^m(\Omega))\cap L^p(0,T;W^{1,p}_0(\Omega)) $$
    with~$m,p\geq 1$ there holds
    $$\iint_{\Omega_T}|u|^l\,\dx\dt \leq C\iint_{\Omega_T}|Du|^p\,\dx\dt \, \bigg(\sup\limits_{t\in(0,T)}\int_{\Omega\times\{t\}}|u|^m\,\dx \bigg)^{\frac{p}{n}},$$
    where~$l\coloneqq \frac{p(n+m)}{n}$ and~$C=C(m,n,p)$. 
\end{mylem} 


The subsequent variant of Sobolev's inequality for convex domains can be inferred from~\cite[Lemma~2.3]{boundaryellipticpq}. 
\begin{mylem} \label{lem:sobolev}
    Let~$\Omega \subset\R^n$ be a bounded and convex set and~$1\leq p<n$. For any~$u\in W^{1,p}(\Omega)$ there holds
    \begin{align*}
        \|u\|_{L^{p^{*}}(\Omega)} \leq C (\diam \Omega)^n |\Omega|^{\frac{1}{n}-1} \bigg(\fint_{\Omega}|Du|^p\,\dx \bigg)^{\frac{1}{p}} + \bigg(\fint_{\Omega}|u|^p\,\dx \bigg)^{\frac{1}{p}},
    \end{align*}
    where~$p^* = \frac{np}{n-p}$, and~$C=C(n,p)$ denotes a positive constant.
\end{mylem}


\section{Gradient estimate for smooth solutions} \label{sec:aprioriest}

In this section, we aim to derive an~\textit{a priori}~$L^\infty$-gradient estimate, assuming that both the solutions as well as the data are smooth. To be precise, we consider smooth, classical solutions
$$u\in C^3(\overline{\Omega}_T\cap Q_{2\rho}(z_0),\R^N)$$
with the property that~$u\equiv 0$ on the lateral boundary~$(\partial\Omega)_T\cap Q_{2\rho}(z_0)$ to a parabolic system of the type
\begin{equation} \label{apriori-equation}
    \partial_t u^i - \divv \big(a(|Du|) Du^i \big) = f^i \qquad\mbox{in~$\Omega_T\cap Q_{2\rho}(z_0)$,\quad for ~$i=1,\ldots,N$},
\end{equation}
where~$\Lambda_{2\rho}(t_0)\subset(0,T)$ and the coefficients~$a$ satisfy assumptions~\eqref{A-Voraussetzung1} --~\eqref{A-Voraussetzung2} with parameter~$\mu\in(0,1]$. The main result of this section is the following proposition, stating a quantitative local~$L^\infty$-gradient estimate. 


\begin{myproposition} [\textit{A priori} $L^\infty$-gradient estimate] \label{prop:apriorigradientsuper}
    Let~$\Omega\subset\R^n$ be a bounded and convex~$C^2$-set and~$Q_{2\rho}(z_0)$ be an arbitrary cylinder with~$\Lambda_{2\rho}(t_0)\subset(0,T)$. Moreover, let~$u\in C^3(\overline{\Omega}_T \cap Q_{2\rho}(z_0),\R^N )$ be a solution to the parabolic system~\eqref{apriori-equation} with
$$u\equiv 0 \qquad\mbox{on the lateral boundary~$(\partial\Omega)_T\cap Q_{2\rho}(z_0)$},$$
where~$2\leq p\leq q < p+\frac{4}{n+2}$, and~$f\in C^1(\overline{\Omega}_T \cap Q_{2\rho}(z_0),\R^N)$ with~$\supp f \Subset\Omega\times\R$. Then, there exists a positive constant~$C=C(C_1,C_2,N,n,p,q,\sigma,\Theta_{\rho}(x_0))\geq 1$, such that the quantitative $L^\infty$-gradient estimate
\begin{align} \label{est:apriorigradientsuper}
  &\sup\limits_{\Omega_T \cap Q_\rho(z_0)} |Du| \\
  & \qquad\,\, \leq C\bigg( \Big( 1+\|f\|^{\frac{(n+2)(n+2+\sigma)}{\sigma}}_{L^{n+2+\sigma}(\Omega_T\cap Q_{2\rho}(z_0))}\rho^{n+2} \Big)  \fiint_{\Omega_T \cap Q_{2\rho}(z_0)} (1+|Du|^2)^{\frac{p}{2}} \,\dx\dt \bigg) ^{\frac{2}{4+n(p-q)(n+2)}} \nonumber 
\end{align} 
holds true.
\end{myproposition}


\subsection{Energy estimate involving second order derivatives} \label{subsec:secorderderivatives}

By following the strategy in~\cite[Proposition~3.2]{boundaryparabolic} and taking Lemma~\ref{lem:ellipticity} into account, we infer the following energy estimate verbatim. 


\begin{mylem} \label{lem:eesecondorder}
Let the assumptions of Proposition~\ref{prop:apriorigradientsuper} hold true. Then, for any non-negative increasing~$\Phi\in C^1(\R_{\geq 0},\R_{\geq 0})$, any~$\phi\in C^\infty_0(B_{2\rho}(x_0))$, and any non-negative~$\eta\in W^{1,\infty}(\Lambda_{2\rho}(t_0),\R_{\geq 0})$, there holds
\begin{align} \label{est:eesecondorder}
    \iint_{ \Omega_T\cap Q_{2\rho}(z_0)} & \phi^2\eta \big(\partial_t \Psi(|Du|) + \Phi(|Du|) (\mu^2+|Du|^2)^{\frac{p-2}{2}}|D^2 u|^2 \big)  \,\dx\dt \\
    & \leq C \iint_{\Omega_T\cap Q_{2\rho}(z_0)} |D\phi|^2 \eta \Phi(|Du|) (\mu^2+|Du|^2)^{\frac{q}{2}}\,\dx\dt \nonumber \\
    &\quad + C \iint_{\Omega_T\cap Q_{2\rho}(z_0)} \phi^2 \eta \Phi(|Du|) Du\cdot Df \,\dx\dt, \nonumber
\end{align}
 with a positive constant~$C=C(C_1,C_2,p,q) \geq 1$, where
\begin{equation} \label{psifunction}
    \Psi(s) \coloneqq \int_{0}^{s} \Phi(r) r\,\dr.
\end{equation}
    
\end{mylem}


\begin{remark} \upshape
The proof of Lemma~\ref{est:eesecondorder} is identical to the one in~\cite[Proposition~3.2]{boundaryparabolic} up to the point where we use the nonstandard~$(p,q)$-growth condition of our system~\eqref{pde}. In fact, the energy estimate from~\cite[Proposition~3.2]{boundaryparabolic} continues to hold true under our nonstandard~$(p,q)$-growth condition. By exploiting the ellipticity and growth condition from Lemma~\ref{lem:ellipticity}, we then establish the energy estimate~\eqref{est:eesecondorder}. We note that the assumed convexity and~$C^2$-regularity of the domain~$\Omega$, as well as the homogeneous Dirichlet boundary condition on the lateral boundary for~$u$, as stated in Lemma~\ref{lem:eesecondorder}, are essential for applying a certain identity from differential geometry that is taken from Grisvard~\cite[(3,1,1,8)]{grisvard2011elliptic}. This result is particularly important as it enables us to discard a boundary integral term due to its non-negative contribution. This idea stems from~\cite[Lemma~2.1 and~(3.14)]{boundaryparabolic} and also~\cite[Lemma~2.2 and~(3.15)]{boundaryellipticpq}.
\end{remark}


\subsection{A reverse Hölder-type inequality} \label{subsec:reverseholder}

In this section, we derive the following inequality of reverse Hölder-type, which is the starting point for the subsequent Moser iteration procedure performed in the following Section~\ref{subsec:moser}. 


\begin{mylem} \label{lem:reverseholder}
    Let the assumptions of Lemma~\ref{lem:eesecondorder} hold true and let~$\alpha\geq 0$ denote an arbitrarily chosen parameter. Then, for any~$\phi\in C^\infty_0(B_{2\rho}(x_0),[0,1])$ and any~$\eta\in W^{1,\infty}(\Lambda_{2\rho}(t_0),[0,1])$ with~$\eta(t_0-4\rho^2)=0$ as well as~$\partial_t\eta\geq 0$, the following reverse Hölder-type inequality holds:
    \begin{align} \label{est:reverseholder}
    & \bigg( \fiint_{\Omega_T\cap Q_{2\rho}(z_0)}  (\phi^2\eta)^{\frac{n+2}{n}}
     (\mu^2+|Du|^2)^{\frac{p+2\alpha+2\delta}{2}} \,\dx\dt \bigg)^{\frac{n}{n+2}} \\
     & \qquad \leq C (p+2\alpha)^{3\frac{n+2+\sigma}{\sigma}} \widetilde{C} \fiint_{\Omega_T\cap Q_{2\rho}(z_0)} \big( (\mu^2+|Du|^2)^{\frac{q+2\alpha}{2}}+1 \big) \,\dx\dt \nonumber
    \end{align}
with~$C=C(C_1,C_2,N,n,p,q,\sigma,\Theta_{\rho}(x_0))$. Here, we abbreviated
\begin{align} \label{est:reverseholderfaktor}
    \widetilde{C} \coloneqq 1 + \rho^2\big(\|D\phi\|_{L^\infty(B_{2\rho}(x_0))} + \|\partial_t\eta\|_{L^\infty(\Lambda_{2\rho}(t_0))} \big) +  \bigg( \rho^{\sigma} \iint_{\Omega_T\cap Q_{2\rho}(z_0)} |f|^{n+2+\sigma}\,\dx\dt \bigg)^{\frac{2}{\sigma}}
\end{align}
and
\begin{align} \label{est:delta}
    \delta \coloneqq \frac{2(\alpha+1)}{n}.
\end{align}
\end{mylem}


\begin{remark} \upshape
   We note that the exponent~$p+2\alpha+2\delta$ on the left-hand side of estimate~\eqref{est:reverseholder} is strictly greater than the exponent~$q+2\alpha$ on the right-hand side, such that the estimate is indeed of reverse Hölder-type. This matter of fact can be verified via
   \begin{align} \label{est:linksrechts}
       p+2\alpha+2\delta &= p+2\alpha+\frac{4(\alpha+1)}{n} = p+2\alpha\Big(1+\frac{2}{n}\Big) + \frac{4}{n} \\
       &= q+2\alpha\Big(1+\frac{2}{n}\Big) + \frac{4}{n} + p-q > q+2\alpha\Big(1+\frac{2}{n}\Big) > q + 2\alpha, \nonumber
   \end{align}
   which follows from the assumption~$q<p+\frac{4}{n+2}$ according to~\eqref{parameter}. Indeed, the weaker assumption
   $$ 2 \leq p\leq q < p+\frac{4}{n}$$
   on the parameters actually turns out to be sufficient at this point. However, for the interpolation argument in Section~\ref{subsec:interpolation}, the stronger assumption~\eqref{parameter} is necessary. 
\end{remark}


\begin{proof}
  In Lemma~\ref{lem:eesecondorder} we choose the non-negative and increasing function~$\Phi$ as follows
  \begin{equation}
      \Phi(s) = \widetilde{\Phi}\Big(\sqrt{\mu^2 + s^2}\Big),\qquad\mbox{where~}\widetilde{\Phi}(r) = r^{2\alpha} \chi^2(r) \quad\mbox{for~$\alpha \geq 0$}.
  \end{equation}
  Here,~$\chi\in C^1(R_{\geq 0},[0,1])$ denotes a smooth cut-off function subject to the properties 
\begin{align*}
    \begin{cases}
        \chi\equiv 0 & \mbox{on~$[0,1]$}, \\
        \chi\equiv 1 & \mbox{on~$[2,\infty)$}, \\
        0\leq \chi'\leq 2 & \mbox{on~$(1,2)$}.
    \end{cases}
\end{align*}
 Additionally, the spatial cut-off function~$\phi$ is chosen in a way such that it is non-negative and bounded by one, i.e. we assume~$\phi\in C^\infty_0(B_{2\rho}(x_0),[0,1])$. By abbreviating~$\G(\xi) = \sqrt{\mu^2+|\xi|^2}$ for~$\xi\in\R^{Nn}$, we thus have that~$ \Phi(|\xi|)=\widetilde{\Phi}(\G(\xi))$. Since~$\mu\in(0,1]$, we bound~$\Psi$ further from above by a change of variables and by calculating
  \begin{align} \label{psiupper}
      \Psi(|Du|) &= \int_{0}^{|Du|} (\mu^2+s^2)^{\alpha} s\chi^2\Big(\sqrt{\mu^2 + s^2}\Big)  \,\ds \\
      &= \int_{\mu}^{\G(Du)} r^{2\alpha+1} \chi^2(r)  \,\dr \nonumber \\
      &\leq \frac{1}{2(\alpha+1)} \G(Du)^{2(\alpha+1)}. \nonumber
  \end{align}
Moreover, again due to~$\mu\in(0,1]$, we also obtain the following lower bound
\begin{equation} \label{psilower}
    \Psi(|Du|) \geq \frac{1}{2(\alpha+1)} (\G(Du)^{2(\alpha+1)}-1). 
\end{equation}
We commence by treating the integral term containing the datum~$f$ on the right-hand side of estimate~\eqref{est:eesecondorder} in Lemma~\ref{lem:eesecondorder}. An integration by parts yields
\begin{align*}
    C & \iint_{\Omega_T\cap Q_{2\rho}(z_0)}  \phi^2 \eta \Phi(|Du|) Du\cdot Df \,\dx\dt \\
    &= - C \sum\limits_{\alpha=1}^{n}\sum\limits_{i=1}^{N} \iint_{\Omega_T\cap Q_{2\rho}(z_0)} \eta f^i \Big[\Phi(|Du|)\phi^2 u^i_{x_\alpha} \Big]_{x_\alpha} \,\dx\dt \\
    &= - C \sum\limits_{\alpha,\beta=1}^{n}\sum\limits_{i,j=1}^{N} \iint_{\Omega_T\cap Q_{2\rho}(z_0)} \eta \phi^2 f^i \bigg(\Phi(|Du|) \delta^{ij}_{\alpha\beta} + \Phi'(|Du|)\frac{u^i_{x_\alpha} u^j_{x_\beta}}{|Du|} \bigg) u^j_{x_\alpha x_\beta} \,\dx\dt \\
    &\quad - 2C \sum\limits_{\alpha=1}^{n}\sum\limits_{i=1}^{N} \iint_{\Omega_T\cap Q_{2\rho}(z_0)} \eta \phi f_i \Phi(|Du|) \phi_{x_\alpha} u^i_{x_\alpha} \,\dx\dt \\
    &\eqqcolon \foo{I} + \foo{II}
\end{align*}
with~$C=C(C_1,C_2,p,q)$. Both integral quantities~$\foo{I}$ and~$\foo{II}$ are bounded further from above, where we obtain for the former
\begin{align*}
    \foo{I} \leq C \iint_{\Omega_T\cap Q_{2\rho}(z_0)} \eta \phi^2 |f| \big(\Phi(|Du|) + \Phi'(|Du|)|Du| \big) |D^2u| \,\dx\dt
\end{align*}
with~$C=C(C_1,C_2,N,n,p,q)$. Due to our construction of~$\Phi$ and~$\chi$,  and also due to the definition of~$\G$, we have that
\begin{align} \label{propertieschi}
   \begin{cases}
       \chi'(\G(Du)) \neq 0 \quad \Longleftrightarrow \quad 1< \G(|Du|) < 2, & \, \\
    \chi(\G(Du)) \neq 0 \quad \,\, \Longleftrightarrow \quad \G(|Du|)>1.
   \end{cases} 
\end{align}
These properties allow us to further estimate 
\begin{align*}
    \Phi(|Du|) & + \Phi'(|Du|)|Du| \\
    &\leq \G(Du)^{2\alpha}\chi(\G(Du)) \big( (1+2\alpha)\chi(\G(Du)) + 2\chi'(\G(Du))|Du| \big) \\
    &\leq C(1+2\alpha)\chi(\G(Du))\G(Du)^{2\alpha} \\
    &\leq C(1+2\alpha)\chi(\G(Du))\G(Du)^{2\alpha + \frac{p-2}{2}}.
\end{align*}
We note that the exponent in the last quantity on the right-hand side of the preceding estimate is always positive according to our assumption~\eqref{parameter}. Exploiting this reasoning in the above estimate for~$\foo{I}$ and also applying Young's inequality, there holds
\begin{align*}
    \foo{I} &\leq C(1+2\alpha) \iint_{\Omega_T\cap Q_{2\rho}(z_0)} \eta \phi^2 |f| \chi(\G(Du))\G(Du)^{2\alpha + \frac{p-2}{2}} |D^2u| \,\dx\dt \\
    &\leq \mbox{$\frac{1}{2}$} \iint_{\Omega_T\cap Q_{2\rho}(z_0)} \eta  \phi^2\chi^2(\G(Du))^{p+2(\alpha-1)}|D^2 u|^2  \,\dx\dt \\
    &\quad + C (1+2\alpha)^2 \iint_{\Omega_T\cap Q_{2\rho}(z_0)} \eta \phi^2 |f|^2 \G(Du)^{2\alpha} \,\dx\dt \\
    &\eqqcolon \mbox{$\frac{1}{2}$} \foo{I}_1 + C(1+2\alpha)^2 \foo{I}_2
\end{align*}
with constant~$C=C(C_1,C_2,N,n,p,q)$. Similarly, the second term~$\foo{II}$ is bounded from above by using that~$q\geq 2$ as follows
\begin{align*}
    \foo{II} &\leq C \iint_{\Omega_T\cap Q_{2\rho}(z_0)} \eta \phi |f| \Phi(|Du|) |D\phi| |Du| \,\dx\dt \\
    &\leq C \iint_{\Omega_T\cap Q_{2\rho}(z_0)} \eta \phi |f| \chi^2(\G(Du))\G(Du)^{2\alpha+1}  |D\phi| \,\dx\dt \\
    &\leq C \iint_{\Omega_T\cap Q_{2\rho}(z_0)} \eta \phi |f| \chi^2(\G(Du))\G(Du)^{2\alpha+\frac{q}{2}}  |D\phi| \,\dx\dt \\
    &\leq \mbox{$\frac{1}{2}$} \foo{I}_2 + C \iint_{\Omega_T\cap Q_{2\rho}(z_0)} \eta \chi^4(\G(Du))
    \G(Du)^{q+2\alpha} |D\phi|^2 \,\dx\dt
\end{align*}
with~$C=C(C_1,C_2,N,n,p,q)$. After reabsorbing the term~$\foo{I}_1$ into the left-hand side of~\eqref{est:eesecondorder} and using the definition of~$\foo{I}_2$ and the preceding upper bound for the remaining quantity~$\foo{II}$, we thus obtain
\begin{align} \label{est:reverseholdereins}
      & \iint_{ \Omega_T  \cap Q_{2\rho}(z_0)} \phi^2\eta \big(\partial_t \Psi(|Du|) + \chi^2(\G(Du)) \G(Du)^{p+2(\alpha-1)} |D^2 u|^2  \big)  \,\dx\dt \\
    &\quad \leq C \iint_{\Omega_T\cap Q_{2\rho}(z_0)} |D\phi|^2 \eta \chi^4(\G(Du)) \G(Du)^{q+2\alpha} \,\dx\dt \nonumber \\
    &\quad \quad + C(1+2\alpha)^2\iint_{\Omega_T\cap Q_{2\rho}(z_0)} \eta \phi^2 |f|^2 \G(Du)^{2\alpha}  \,\dx\dt \nonumber
\end{align}
with~$C$ exhibiting the same dependence as before. As a next step, we treat the term involving the time derivative in~\eqref{est:reverseholdereins}. Thereby, we choose~$\eta = \eta_1\eta_2$, where~$\eta_1 \in W^{1,\infty}(\Lambda_{2\rho}(t_0),[0,1])$ is a cut-off function satisfying~$\eta_1(t_0-4\rho^2)=0$ and~$\partial_t \eta_1 \geq 0$, while~$\eta_2 \in W^{1,\infty}([t_0-4\rho^2,t_0],[0,1])$ is given by
\begin{align*}
    \eta_2(t) \coloneqq \begin{cases}
        1 & \mbox{for $t\in[t_0-4\rho^2,\tau]$}, \\
        1-\frac{t-\tau}{\epsilon} & \mbox{for $t\in(\tau,\tau+\epsilon]$}, \\
        0 & \mbox{for $t\in[\tau+\epsilon,t_0]$},
    \end{cases}
\end{align*}
for arbitrarily chosen parameters~$\epsilon>0,\tau>0$ subject to the relation~$t_0 -4\rho^2<\tau<\tau+\epsilon<t_0$. An integration by parts therefore yields
\begin{align*}
    &\iint_{ \Omega_T  \cap Q_{2\rho}(z_0)} \phi^2\eta_1 \eta_2 \partial_t \Psi(|Du|) \,\dx\dt \\
    & \quad= - \iint_{ \Omega_T  \cap Q_{2\rho}(z_0)} \phi^2[\partial_t \eta_1\, \eta_2 + \eta_1 \partial_t \eta_2] \Psi(|Du|) \,\dx\dt \\
    &\quad= - \iint_{ \Omega_T  \cap Q_{2\rho}(z_0)} \phi^2 \partial_t \eta_1\, \eta_2 \Psi(|Du|) \,\dx\dt + \frac{1}{\epsilon} \iint_{ \Omega_T  \cap B_{2\rho}(x_0)\times(\tau,\tau+\epsilon)} \phi^2\eta_1 \Psi(|Du|) \,\dx\dt
\end{align*}
for any~$t_0 -4\rho^2<\tau<\tau+\epsilon<t_0$. Passing to the limit~$\epsilon\downarrow 0$ and exploiting estimate~\eqref{psiupper}, we further obtain for any~$t_0 -4\rho^2<\tau<t_0$ the bound
\begin{align*}
     &\int_{\Omega \cap B_{2\rho}(x_0) \times\{\tau\}} \phi^2\eta_1 \Psi(|Du|) \,\dx\dt \\
     &\quad + \iint_{ \Omega_T  \cap B_{2\rho}(x_0)\times(t_0-4\rho^2,\tau)} \phi^2 \eta_1 \chi^2(\G(Du)) \G(Du)^{p+2(\alpha-1)} |D^2 u|^2  \,\dx\dt \\
    &\leq C \iint_{\Omega_T\cap Q_{2\rho}(z_0)} |D\phi|^2 \eta \chi^4(\G(Du)) \G(Du)^{q+2\alpha} \,\dx\dt \\
    &\quad + C(1+2\alpha)^2\iint_{\Omega_T\cap Q_{2\rho}(z_0)} \eta \phi^2 |f|^2 \G(Du)^{2\alpha}  \,\dx\dt \\
    &\quad + \frac{C}{2(\alpha+1)} \iint_{\Omega_T\cap Q_{2\rho}(z_0)} \phi^2 \partial_t \eta_1\G(Du)^{2(\alpha+1)} \,\dx\dt. 
\end{align*}
Since~$\tau\in(t_0 -4\rho^2,t_0)$ is arbitrarily, we first pass to the limit~$\tau\uparrow t_0$ in the second integral on the left hand-side of the preceding estimate, while in the first term we take the supremum with respect to~$\tau\in(t_0 -4\rho^2,t_0)$. This way, we achieve
\begin{align*}
     \sup\limits_{\tau\in\Lambda_{2\rho}(t_0)} & \int_{\Omega \cap B_{2\rho}(x_0) \times\{\tau\}} \phi^2\eta_1 \Psi(|Du|) \,\dx\dt \\
     & + \iint_{ \Omega_T  \cap Q_{2\rho}(z_0)} \phi^2 \eta_1 \chi^2(\G(Du)) \G(Du)^{p+2(\alpha-1)} |D^2 u|^2  \,\dx\dt \\
    &\leq C \iint_{\Omega_T\cap Q_{2\rho}(z_0)} |D\phi|^2 \eta_1 \chi^4(\G(Du)) \G(Du)^{q+2\alpha} \,\dx\dt \\
    &\quad + C(1+2\alpha)^2\iint_{\Omega_T\cap Q_{2\rho}(z_0)} \eta_1 \phi^2 |f|^2 \G(Du)^{2\alpha}  \,\dx\dt \\
    &\quad + \frac{C}{2(\alpha+1)} \iint_{\Omega_T\cap Q_{2\rho}(z_0)} \phi^2 \partial_t \eta_1\G(Du)^{2(\alpha+1)} \,\dx\dt 
\end{align*}
with a constant~$C$ as before. Due to~\eqref{psilower}, the sup-term on the left-hand side may be bounded further below and the~$-1$ is shifted onto the right-hand side. Moreover, we multiply both sides of the inequality by~$p+2\alpha$ and take mean values. As a result, we achieve the estimate 
\begin{align} \label{est:reverseholderzwei}
      &\sup\limits_{\tau\in\Lambda_{2\rho}(t_0)}~\fint_{\Omega \cap B_{2\rho}(x_0) \times\{\tau\}} \phi^2\eta_1 \G(Du)^{2(\alpha+1)} \,\dx\dt \\
     & \quad + (p+2\alpha)\rho^2 \fiint_{ \Omega_T  \cap Q_{2\rho}(z_0)} \phi^2 \eta_1 \chi^2(\G(Du)) \G(Du)^{p+2(\alpha-1)} |D^2 u|^2  \,\dx\dt \nonumber \\
    &\leq C (p+2\alpha) \rho^2 \fiint_{\Omega_T\cap Q_{2\rho}(z_0)} \big( |D\phi|^2 \chi^4(\G(Du))  \G(Du)^{q+2\alpha} + \partial_t\eta_1\, \G(Du)^{2(\alpha+1)} \big) \,\dx\dt \nonumber \\
     &\quad + C(1+2\alpha)^3 \rho^2 \iint_{\Omega_T\cap Q_{2\rho}(z_0)} \eta_1 \phi^2 |f|^2 \G(Du)^{2\alpha}  \,\dx\dt + C \nonumber \\
     & \leq C(p+2\alpha) \foo{III}_1 + C(p+2\alpha)^3 \foo{III}_2 + C \eqqcolon \foo{III} \nonumber 
\end{align} 
with~$C=C(C_1,C_2,N,n,p,q)$, where we abbreviated
$$ \foo{III}_1 \coloneqq \rho^2 \fiint_{\Omega_T\cap Q_{2\rho}(z_0)} \big( \|D\phi\|^2_{L^\infty(B_{2\rho}(x_0))} \chi^4(\G(Du))  \G(Du)^{q+2\alpha} + \partial_t\eta_1\, \G(Du)^{2(\alpha+1)} \big) \,\dx\dt $$
as well as
$$ \foo{III}_2 \coloneqq \rho^2 \iint_{\Omega_T\cap Q_{2\rho}(z_0)} \eta_1 \phi^2 |f|^2 \G(Du)^{2\alpha}  \,\dx\dt. $$

Next, we aim to perform an interpolation argument. We recall the quantity~$\delta=\frac{2(\alpha+1)}{n}$ defined in~\eqref{est:delta}. The particular choice of this parameter becomes clear in the course of the proof. A calculation yields for any~$(x,t)\in (\Omega_T\cap Q_{2\rho}(z_0))$ the following
\begin{align*}
    \Big|  D \Big[&\phi^{\frac{n+2}{n}}  \chi^2(\G(Du))\G(Du)^{\frac{p+2\alpha+2\delta}{2}} \Big]\Big| \\ 
    &= \Big|\mbox{$\frac{n+2}{n}$} \phi^{\frac{2}{n}} D\phi \chi^2(\G(Du))\G(Du)^{\frac{p+2\alpha+2\delta}{2}} \\
    &\quad + 2 \phi^{\frac{n+2}{n}} \chi(\G(Du))\chi'(\G(Du)) \G(Du)^{-1} |Du| D|Du| \\
    &\quad + \mbox{$\frac{p+2\alpha+2\delta}{2}$} \phi^{\frac{n+2}{n}} \chi^2(\G(Du)) \G(Du)^{\frac{p+2\alpha+2\delta-4}{2}} |Du| D|Du| \Big| \\
    &\leq C(p+2\alpha) \phi^{\frac{n+2}{n}} \chi(\G(Du)) \G(Du)^{\frac{p+2\alpha+2\delta-2}{2}} |D^2u| \\
    &\quad + C |D\phi| \phi^{\frac{2}{n}} \chi^2(\G(Du)) \G(Du)^{\frac{p+2\alpha+2\delta}{2}},
\end{align*}
where we used the properties~\eqref{propertieschi}. Instead of estimating the factor~$\chi^2(\G(Du))$ by one, we keep this term in the last estimate. The constant~$C$ in the previous estimate only depends on the space dimension~$n$. We now exploit the previous calculation and employ the Sobolev inequality from Lemma~\ref{lem:sobolev}. Hereby, we note that, according to the reasoning provided in Section~\ref{subsec:convex}, the Sobolev constant from the same lemma can be controlled from above by a constant only depending on~$n$ and the geometric quantity~$\Theta_{\rho}(x_0)$. Additionally, since~$\phi\in C^\infty_0(B_{2\rho}(x_0))$ is smooth and compactly supported within~$B_{2\rho}(x_0)$, we readily obtain by Taylor's theorem the estimate
$$\|\phi\|_{L^\infty(B_{2\rho}(x_0))}\leq \rho \|D\phi\|_{L^\infty(B_{2\rho}(x_0))}.$$
Taking these preliminary considerations into account, we apply Lemma~\ref{lem:sobolev} with parameter~$p=\frac{2n}{n+2}$, which yields~$p^* = 2$, to achieve for any~$t\in \Lambda_{2\rho}(t_0)$ the following inequality
\begin{align*}
    &\fint_{\Omega\cap B_{2\rho}(x_0)} \phi^{\frac{2(n+2)}{n}} \chi^4(\G(Du))\G(Du)^{p+2\alpha+2\delta}\,\dx \\
    &\leq C \rho^2 \bigg(\fint_{\Omega\cap B_{2\rho}(x_0)} \Big| D \Big[\phi^{\frac{n+2}{n}} \chi^2(\G(Du))\G(Du)^{\frac{p+2\alpha+2\delta}{2}} \Big]\Big|^{\frac{2n}{n+2}}  \,\dx \bigg)^{\frac{n+2}{n}} \\
    &\quad + 2 \bigg(\fint_{\Omega\cap B_{2\rho}(x_0)} \Big| \phi^{\frac{n+2}{n}} \chi^2(\G(Du))\G(Du)^{\frac{p+2\alpha+2\delta}{2}}\Big|^{\frac{2n}{n+2}}  \,\dx \bigg)^{\frac{n+2}{n}} \\
    &\leq C \rho^2 (p+2\alpha)^2 \bigg(\fint_{\Omega\cap B_{2\rho}(x_0)} \Big| \phi^{\frac{n+2}{n}} \chi(\G(Du)) \G(Du)^{\frac{p+2\alpha+2\delta-2}{2}} |D^2u|\Big|^{\frac{2n}{n+2}}  \,\dx \bigg)^{\frac{n+2}{n}} \\
    &\quad + C \rho^2 \|D\phi\|^2_{L^\infty(B_{2\rho}(x_0))} \bigg(\fint_{\Omega\cap B_{2\rho}(x_0)} \Big| \phi^{\frac{2}{n}} \chi^2(\G(Du))\G(Du)^{\frac{p+2\alpha+2\delta}{2}}\Big|^{\frac{2n}{n+2}}  \,\dx \bigg)^{\frac{n+2}{n}},  
\end{align*}
with a constant~$C=C(C_1,C_2,N,n,p,q,\Theta_{\rho}(x_0))$. For both integral terms on the right-hand side of the preceding inequality, we apply Hölder's inequality with exponents~$(\frac{n+2}{2},\frac{n+2}{n})$ to the functions
\begin{align*}
\begin{cases}
    \Big|\phi^{\frac{n+2}{n}} \G(Du)^{\delta} |D^2u|\Big|^{\frac{2n}{n+2}}, \quad \Big| \chi(\G(Du)) \G(Du)^{\frac{p+2\alpha-2}{2}} |D^2u|\Big|^{\frac{2n}{n+2}}, & \\
    \Big| \phi^{\frac{2}{n}} \G(Du)^{\delta}\Big|^{\frac{2n}{n+2}}, \quad   \Big| \chi^2(\G(Du))\G(Du)^{\frac{p+2\alpha}{2}}\Big|^{\frac{2n}{n+2}}. &
    \end{cases}
\end{align*}
This way, we obtain
\begin{align} \label{est:reverseholderdrei}
   & \fint_{\Omega\cap B_{2\rho}(x_0)} \phi^{\frac{2(n+2)}{n}} 
   \G(Du)^{p+2\alpha+2\delta}\,\dx \leq C \rho^2\foo{IV}_1 \cdot\foo{IV}^{\frac{2}{n}}_2 
\end{align}
with~$C=C(C_1,C_2,N,n,p,q,\Theta_{\rho}(x_0))$, where 
\begin{align*}
     \foo{IV}_1 &\coloneqq (p+2\alpha)^2  \fint_{\Omega\cap B_{2\rho}(x_0)} \phi^{2} \chi^2(\G(Du)) \G(Du)^{p+2\alpha-2}|D^2u|^2\,\dx \\
     & \qquad + \|D\phi\|^2_{L^\infty(B_{2\rho}(x_0))} \fint_{\Omega\cap B_{2\rho}(x_0)} \chi^4(\G(Du)) \G(Du)^{p+2\alpha} \,\dx 
\end{align*}
and
$$ \foo{IV}_2 \coloneqq \fint_{\Omega\cap B_{2\rho}(x_0)} \phi^2 \G(Du)^{n\delta}\,\dx. $$

Due to our choice of~$\delta=\frac{2(\alpha+1)}{n}$, the exponent in the integral~$\foo{IV}_2$ equals the integrability exponent that is present in the sup-term of the energy estimate~\eqref{est:reverseholderzwei}. Therefore, we multiply both sides of the preceding inequality by~$\eta^{\frac{n+2}{n}}_1$ and take mean values with respect to~$t\in\Lambda_{2\rho}(t_0)$. Moreover, we apply the energy estimate~\eqref{est:reverseholderzwei}, which yields
\begin{align} \label{est:reverseholdervier}
    \fiint_{\Omega_T\cap Q_{2\rho}(z_0)} & (\phi^2\eta_1)^{\frac{n+2}{n}}
    \G(Du)^{p+2\alpha+2\delta} \,\dx\dt \\
    &\leq C\rho^2 \fint_{\Lambda_{2\rho}(t_0)}  \eta_1 \foo{IV}_1 \cdot (\eta_1 \foo{IV}_2)^{\frac{2}{n}} \,\dt \nonumber \\
    &\leq C \rho^2\bigg( \sup\limits_{\tau\in\Lambda_{2\rho}(t_0)}~\fint_{\Omega \cap B_{2\rho}(x_0) \times\{\tau\}} \phi^2\eta_1 \G(Du)^{2(\alpha+1)} \,\dx\dt\bigg)^{\frac{2}{n}} \fint_{\Lambda_{2\rho}(t_0)} \eta_1 \foo{IV}_1 \,\dt \nonumber \\
    &\leq C\rho^2 \foo{III}^{\frac{2}{n}} \fint_{\Lambda_{2\rho}(t_0)} \eta_1 \foo{IV}_1 \,\dt \nonumber
\end{align}
with~$C=C(C_1,C_2,N,n,p,q,\Theta_{\rho}(x_0))$. Exploiting the properties~\eqref{propertieschi} once more, we bound the right-hand side of the preceding estimate further from above by
\begin{align*}
    \fint_{\Lambda_{2\rho}(t_0)} \eta_1 \foo{IV}_1 \,\dt &\leq (p+2\alpha)^2  \fiint_{\Omega_T\cap Q_{2\rho}(z_0)} \phi^{2} \eta_1 \chi^2(\G(Du)) \G(Du)^{p+2\alpha-2}|D^2u|^2\,\dx\dt \\
    &\quad + \|D\phi\|^2_{L^\infty(B_{2\rho}(x_0))} \fiint_{\Omega_T\cap Q_{2\rho}(z_0)} \eta_1 \chi^4(\G(Du)) \G(Du)^{p+2\alpha} \,\dx\dt \\
    &\leq (p+2\alpha)^2  \fiint_{\Omega_T \cap Q_{2\rho}(z_0)} \phi^{2} \eta_1 \chi^2(\G(Du)) \G(Du)^{p+2(\alpha-1)}|D^2u|^2\,\dx\dt \\
    &\quad + \|D\phi\|^2_{L^\infty(B_{2\rho}(x_0))} \fiint_{\Omega_T\cap Q_{2\rho}(z_0)} \chi^4(\G(Du)) \G(Du)^{q+2\alpha} \,\dx\dt \\
    &\leq C \frac{(p+2\alpha)}{\rho^2} \foo{III}.
\end{align*}
Therefore, the left-hand side in~\eqref{est:reverseholderdrei} can be estimated as follows
\begin{align} \label{est:reverseholderfünf}
     \bigg( \fiint_{\Omega_T\cap Q_{2\rho}(z_0)}  (\phi^2\eta_1)^{\frac{n+2}{n}} 
     \G(Du)^{p+2\alpha+2\delta} \,\dx\dt \bigg)^{\frac{n}{n+2}} \leq C(p+2\alpha)\foo{III}
\end{align}
with constant~$C=C(C_1,C_2,N,n,p,q,\Theta_{\rho}(x_0))$. \,\\

We now treat the quantity $\foo{III}_2$ present in $\foo{III}$. Here, we employ Hölder's inequality with exponents $(\frac{n+2+\sigma}{2},\frac{n+2+\sigma}{n+\sigma})$. Additionally, we recall the comments made in Section \ref{subsec:convex}, allowing us to use the bound $\frac{\rho^{n+2}}{|\Omega_T\cap Q_{2\rho}(z_0)|} \leq C(n) \Theta_{\rho}(x_0)$. This way, we obtain
\begin{align} \label{est:fabkürz}
    \foo{III}_2 &\leq \Theta^{\frac{2}{n+2+\sigma}}_{\rho}(x_0) \underbrace{\bigg( \rho^{\sigma} \iint_{\Omega_T\cap Q_{2\rho}(z_0)} |f|^{n+2+\sigma}\,\dx\dt \bigg)^{\frac{2}{n+2+\sigma}}}_{ \eqqcolon  [[f]]^2_{n+2+\sigma;\Omega_T\cap Q_{2\rho}(z_0)}} \\
    &\quad \cdot\bigg( \fiint_{\Omega_T\cap Q_{2\rho}(z_0)} \big( \eta_1\phi^2 \G(Du)^{2\alpha} \big)^{\frac{n+2+\sigma}{n+\sigma}}\,\dx\dt \bigg)^{\frac{n+\sigma}{n+2+\sigma}} . \nonumber
\end{align}
Since the integrability exponent differs from the one in the left-hand side of~\eqref{est:reverseholderdrei}, we use an interpolation argument. In fact, we first employ Lemma~\ref{lem:interpolation} with parameters~$p=1, q=\frac{n+2+\sigma}{n+\sigma},r=\frac{n+2}{n}$, and then use Young's inequality, which yields for any~$\epsilon>0$ the following estimate
\begin{align*}
    \foo{III}_2 &\leq \Theta^{\frac{2}{n+2+\sigma}}_{\rho}(x_0) [[f]]^2_{n+2+\sigma;\Omega_T\cap Q_{2\rho}(z_0)} \bigg( \epsilon 
    \bigg( \fiint_{\Omega_T\cap Q_{2\rho}(z_0)} \big( \eta_1\phi^2 \G(Du)^{2\alpha} \big)^{\frac{n+2}{n}}\,\dx\dt \bigg)^{\frac{n+2}{n}} \\
    &\quad + C\epsilon^{-\frac{n+2}{\sigma}} \fiint_{\Omega_T\cap Q_{2\rho}(z_0)} \eta_1\phi^2 \G(Du)^{2\alpha} \,\dx\dt
    \bigg)
\end{align*}
with $C=C(n,p,\sigma)$. In order to estimate the right-hand side from above, we exploit the fact that
$$2\alpha \leq q+2\alpha, \qquad 2\alpha \frac{n+2}{n}\leq q+2\alpha+2\delta $$
and split the domain of integration into two subsets as follows: the first subset consists of points where $\G(Du)> 1$ holds true, allowing us to use the above identities, while on its compliment we have $\G(Du)\leq 1$, which enables us to bound the entire integrand by one. By additionally exploiting the bounds for $\eta_1$ and $\phi$, we thus achieve 
\begin{align*} 
        &\foo{III}_2  \\
        &\quad\leq \Theta^{\frac{2}{n+2+\sigma}}_{\rho}(x_0) [[f]]^2_{n+2+\sigma;\Omega_T\cap Q_{2\rho}(z_0)} \bigg( \epsilon 
    \bigg( \fiint_{\Omega_T\cap Q_{2\rho}(z_0)} \big( (\eta_1\phi^2)^{\frac{n+2}{n}} \G(Du)^{p+2\alpha+2\delta} + 1 \big)\,\dx\dt \bigg)^{\frac{n}{n+2}} \nonumber \\
    &\quad\quad + C\epsilon^{-\frac{n+2}{\sigma}} \fiint_{\Omega_T\cap Q_{2\rho}(z_0)} \big( \G(Du)^{q+2\alpha}+1 \big) \,\dx\dt
    \bigg). \nonumber
\end{align*}
By combining the estimate~\eqref{est:reverseholderfünf} and the preceding estimate for~$\foo{III}_2$, we infer
\begin{align*}
    & \bigg( \fiint_{\Omega_T\cap Q_{2\rho}(z_0)}  (\phi^2\eta_1)^{\frac{n+2}{n}} 
     \G(Du)^{p+2\alpha+2\delta} \,\dx\dt \bigg)^{\frac{n}{n+2}} \\
     &\leq  \epsilon C (p+2\alpha)^3 \Theta^{\frac{2}{n+2+\sigma}}_{\rho}(x_0) [[f]]^2_{n+2+\sigma;\Omega_T\cap Q_{2\rho}(z_0)} \bigg( \fiint_{\Omega_T\cap Q_{2\rho}(z_0)}  (\phi^2\eta_1)^{\frac{n+2}{n}}
     \G(Du)^{p+2\alpha+2\delta} \,\dx\dt \bigg)^{\frac{n}{n+2}}   \\
     &\quad+ C (p+2\alpha)^3\Theta^{\frac{2}{n+2+\sigma}}_{\rho}(x_0) [[f]]^2_{n+2+\sigma;\Omega_T\cap Q_{2\rho}(z_0)} \epsilon^{-\frac{n+2}{\sigma}} \fiint_{\Omega_T\cap Q_{2\rho}(z_0)} \big( \G(Du)^{q+2\alpha}+1 \big) \,\dx\dt \\
     & \quad + C (p+2\alpha) \foo{III}_1  + C
\end{align*}
with a constant~$C=C(C_1,C_2,N,n,p,q,\Theta_{\rho}(x_0))$. At this point, the free parameter~$\epsilon>0$ is chosen in such a way that there holds
$$\epsilon C (p+2\alpha)^3 \Theta^{\frac{2}{n+2+\sigma}}_{\rho}(x_0) [[f]]^2_{n+2+\sigma;\Omega_T\cap Q_{2\rho}(z_0)} = \mbox{$\frac{1}{2}$},  $$
allowing us to re-absorb the first term on the right-hand side of the previous estimate into the left-hand side. The second quantity on the right-hand side, which also involves the parameter~$\epsilon$, then yields
\begin{align*}
    \mbox{$\frac{1}{2}$} \epsilon^{-\frac{n+2}{\sigma}-1} &= (p+2\alpha)^3\Theta^{\frac{2}{n+2+\sigma}}_{\rho}(x_0) [[f]]^2_{n+2+\sigma;\Omega_T\cap Q_{2\rho}(z_0)} \epsilon^{-\frac{n+2}{\sigma}} \\
    &= C\big( (p+2\alpha)^3 
    [[f]]^2_{n+2+\sigma;\Omega_T\cap Q_{2\rho}(z_0)}\big)^{\frac{n+2+\sigma}{\sigma}}
\end{align*}
for some constant~$C=C(C_1,C_2,N,n,p,q,\sigma,\Theta_{\rho}(x_0))$. Thus, the preceding estimate overall turns into
\begin{align*}
    & \bigg( \fiint_{\Omega_T\cap Q_{2\rho}(z_0)}  (\phi^2\eta_1)^{\frac{n+2}{n}}
     \G(Du)^{p+2\alpha+2\delta} \,\dx\dt \bigg)^{\frac{n}{n+2}} \\
     & \quad \leq C (p+2\alpha)^{3\frac{n+2+\sigma}{\sigma}} \bigg(\foo{III}_1 + [[f]]^{2\frac{n+2+\sigma}{\sigma}}_{n+2+\sigma;\Omega_T\cap Q_{2\rho}(z_0)} \fiint_{\Omega_T\cap Q_{2\rho}(z_0)} \big( \G(Du)^{q+2\alpha}+1 \big) \,\dx\dt  +1 \bigg)
\end{align*}
with~$C=C(C_1,C_2,N,n,p,q,\sigma,\Theta_{\rho}(x_0))$. Finally, we employ Young's inequality in the term~$\foo{III}_1$ as well as the estimate~$\chi(\G(Du))\leq 1$, which implies the claimed reverse Hölder inequality as stated in~\eqref{est:reverseholder}.
\end{proof}


\subsection{Moser iteration procedure} \label{subsec:moser}

As before, we use the abbreviation~$\G(Du) = (\mu^2+|Du|^2)^{\frac{1}{2}}$. By rewriting inequality~\eqref{est:reverseholder} according to~\eqref{est:linksrechts}, there holds
    \begin{align} \label{est:iterationeins}
    &  \fiint_{\Omega_T\cap Q_{2\rho}(z_0)}  (\phi^2\eta)^{\frac{n+2}{n}} 
     \G(Du)^{q+2\alpha (1+\frac{2}{n} ) + \frac{4}{n} + p-q} \,\dx\dt \\
     & \quad \leq C \bigg( (p+2\alpha)^{3\frac{n+2+\sigma}{\sigma}} \widetilde{C} \fiint_{\Omega_T\cap Q_{2\rho}(z_0)} \big( \G(Du) ^{q+2\alpha}+1 \big) \,\dx\dt \bigg)^{\frac{n+2}{n}} \nonumber.
    \end{align}
 We now define a sequence of parameters~$(\alpha_k)_{k\in\N_0}$ with~$\alpha_0 \coloneqq 0$ and
 $$ \alpha_k \coloneqq 2 \alpha_{k-1} \Big(1+\frac{2}{n} \Big) + \frac{4}{n} 
+ p - q $$
 for any~$k\in\N$. Then, it follows inductively that there holds
 $$\alpha_k = \Big( 1+\frac{n(p-q)}{4} \Big) \Big( \Big(1+\frac{2}{n} \Big)^k - 1 \Big) $$
for any~$k\in\N_0$. Further rewriting~\eqref{est:iterationeins}, we achieve
\begin{align} \label{est:iterationzwei}
     &  \fiint_{\Omega_T\cap Q_{2\rho}(z_0)}  (\phi^2\eta)^{\frac{n+2}{n}}
     \G(Du)^{q+2\alpha_k} \,\dx\dt \\
     & \quad \leq \bigg( C (p+2\alpha)^{3\frac{n+2+\sigma}{\sigma}} \widetilde{C} \fiint_{\Omega_T\cap Q_{2\rho}(z_0)} \big( \G(Du) ^{q+2\alpha_{k-1}}+1 \big) \,\dx\dt \bigg)^{\frac{n+2}{n}} \nonumber.
\end{align}
Next, we consider arbitrarily chosen radii~$r,s$, such that~$\rho\leq r<s\leq 2\rho$, and set
$$ \rho_k \coloneqq r + \frac{(s-r)}{2^k},\qquad \Q_k \coloneqq Q_{\rho_k}(z_0) $$
for any $k\in\N_0$. Moreover, we specify the cut-off functions $\phi=\phi_k$ and $\eta=\eta_k$. The spatial cut-off functions $\phi_k\in C^\infty_0(B_{\rho_k}(x_0),[0,1])$ are chosen in such a way that $\phi_k \equiv 1$ on~$B_{\rho_{k+1}}(x_0)$ and the standard estimate $|D\phi_k| \leq \frac{2^{k+2}}{s-r}$ is satisfied. The cut-off functions in time $\eta_k\in W^{1,\infty}(\Lambda_{\rho_k}(t_0))$ are chosen adequately so that there holds $\eta_k(t_0-\rho^2_k)=0$, as well as $\eta_k \equiv 1$ on $\Lambda_{\rho_{k+1}}(t_0)$, and also~$0\leq \partial_t\eta_k\leq \frac{2^{2(k+2)}}{(s-r)^2}$. By recalling the constant $\widetilde{C}$ from Lemma \ref{lem:reverseholder} and by exploiting the fact that $\alpha_{k-1} < \alpha_{k}$, we thus obtain
\begin{align} \label{est:iterationdrei}
     &  \frac{|\Omega_T \cap \Q_{k}|}{ |\Omega_T \cap \Q_{k-1}| } \fiint_{\Omega_T\cap \Q_{k}}  (\phi^2\eta)^{\frac{n+2}{n}}
     \G(Du)^{q+2\alpha_k} \,\dx\dt \\
     & \quad \quad \leq \bigg( 4^k  \delta \rho^2 \frac{(1+\alpha_k)^{3\frac{n+2+\sigma}{\sigma}}}{(s-r)^2}  \fiint_{\Omega_T\cap \Q_{k-1}} \big( \G(Du) ^{q+2\alpha_{k-1}}+1 \big) \,\dx\dt \bigg)^{\frac{n+2}{n}} \nonumber,
\end{align}
with~$C=C(C_1,C_2,N,n,p,q,\sigma, \Theta_\rho(x_0))$, where~$\delta \coloneqq C \big([[f]]^{2\frac{n+2+\sigma}{\sigma}}_{n+2+\sigma;\Omega_T\cap Q_{2\rho}(z_0)} + 1 \big)$. The factor on the left-hand side of the preceding estimate~\eqref{est:iterationdrei} is controlled from below by
\begin{align*}
    \frac{C}{\Theta_\rho(x_0)} \leq \frac{|\Omega_T \cap Q_{\rho}(x_0)|}{|Q_{2\rho}(x_0)|} \leq \frac{|\Omega_T \cap \Q_{k}|}{ |\Omega_T \cap \Q_{k-1}| }
\end{align*}
for some constant~$C=C(n)$. Next, for any~$k\in\N$, we define quantities
$$ A_k \coloneqq \fiint_{\Q_k} \G(Du)^{q+2\alpha_k} \,\dx\dt. $$
 The estimate~\eqref{est:iterationdrei} thus yields
\begin{align*}
    A_k \leq \bigg( 4^k  \delta \rho^2 \frac{(p+2\alpha)^{3\frac{n+2+\sigma}{\sigma}}}{(s-r)^2} \bigg)^{\frac{n+2}{n}}  (A_{k-1}+1)^{\frac{n+2}{n}}
\end{align*}
for all~$k\in\N$, where~$\delta>0$ is a constant with the same structure as in~\eqref{est:iterationdrei}. This inequality can be iterated to obtain
\begin{align*}
    A_k &\leq \bigg( 4^k  \delta \rho^2 \frac{(1+\alpha_k)^{3\frac{n+2+\sigma}{\sigma}}}{(s-r)^2} \bigg)^{\frac{n+2}{n}} \bigg( \bigg(  4^{k-1}  \delta \rho^2 \frac{(1+\alpha_{k-1})^{3\frac{n+2+\sigma}{\sigma}}}{(s-r)^2} \bigg)^{\frac{n+2}{n}} (A_{k-2}+1)+1 \bigg)^{\frac{n+2}{n}} \\
    &\leq \bigg( \bigg(  4^{k-1}  2\delta \rho^2 \frac{(1+\alpha_{k-1})^{3\frac{n+2+\sigma}{\sigma}}}{(s-r)^2} \bigg)^{\frac{n+2}{n}} (A_{k-2}+1)\bigg)^{\frac{n+2}{n}} \\
    &\leq \prod_{j=1}^{k} \bigg(  4^{j}  2\delta \rho^2 \frac{(1+\alpha_{j})^{3\frac{n+2+\sigma}{\sigma}}}{(s-r)^2} \bigg)^{(\frac{n+2}{n})^{k-j+1}}  (A_0 + 1)^{(\frac{n+2}{n})^k}
\end{align*}
for any~$k\in\N$. After taking both sides to the power~$\frac{1}{q+2\alpha_k}$, there holds
\begin{align} \label{est:iterationvier}
    A^{\frac{1}{q+2\alpha_k}}_k \leq \prod_{j=1}^{k} \bigg(  4^{j}  2\delta \rho^2 \frac{(1+\alpha_{j})^{3\frac{n+2+\sigma}{\sigma}}}{(s-r)^2} \bigg)^{ \frac{ (\frac{n+2}{n})^{k-j+1}}{q+2\alpha_k}}  (A_0 + 1)^{ \frac{(\frac{n+2}{n})^k}{q+2\alpha_k}}.
\end{align}
Regarding the exponent in the last quantity, we calculate
\begin{align*}
    \lim\limits_{k\to\infty} \frac{(\frac{n+2}{n})^k}{q+2\alpha_k} = \lim\limits_{k\to\infty} \frac{(\frac{n+2}{n})^k}{q + 2\big( 1+\frac{n(p-q)}{4} \big) \Big( \Big(1+\frac{2}{n} \Big)^k - 1 \Big) } = \frac{1}{2 + \frac{n(p-q)}{2}} = \frac{2}{4+n(p-q)}.
\end{align*}
Consequently, by passing to the limit, we have
\begin{align*}
    \lim\limits_{k\to\infty} (A_0 + 1)^{ \frac{(\frac{n+2}{n})^k}{q+2\alpha_k}} = \bigg( \fiint_{\Omega_T\cap \Q_s} ( \G(Du)^{q} + 1 )\,\dx\dt \bigg) ^{\frac{2}{4+n(p-q)}}
\end{align*}
for the second factor on the right-hand side of~\eqref{est:iterationvier}. For the treatment of the first factor on the right-hand side of~\eqref{est:iterationvier}, we abbreviate
$$0 < \kappa \coloneqq \frac{4+n(p-q)}{2}  \leq 2, $$
where the lower and upper bound are a consequence of assumption \eqref{parameter}. Thus, we may rewrite and estimate $\alpha_k$ as follows
$$ \alpha_k = \frac{\kappa}{2} \Big(\Big(\frac{n+2}{n}\Big)^k -1 \Big) \leq \Big(\frac{n+2}{n}\Big)^k -1  $$
for any~$k\in\N$. By further denoting
$$ \widetilde{\kappa} \coloneqq \frac{2\delta\rho^2}{(s-r)^2} = \frac{2C \rho^2}{(s-r)^2} \big([[f]]^{2\frac{n+2+\sigma}{\sigma}}_{n+2+\sigma;\Omega_T\cap Q_{2\rho}(z_0)} + 1 \big) $$
and also
$$  \beta \coloneqq 4 \Big( \frac{n+2}{n} \Big)^{3 \frac{n+2+\sigma}{\sigma}} \geq 1, $$
we therefore have that
$$ 4^j 2\delta \rho^2 \frac{(1+\alpha_j)^{3 \frac{n+2+\sigma}{\sigma}}}{(s-r)^2} \leq \widetilde{\kappa} \beta^j  $$
for any~$j=1,\ldots,k$. This allows us to bound the first factor on the right-hand side of~\eqref{est:iterationvier} by
\begin{align*}
    \prod_{j=1}^{k} \bigg(  4^{j}  2\delta \rho^2 \frac{(1+\alpha_{j})^{3\frac{n+2+\sigma}{\sigma}}}{(s-r)^2} \bigg)^{ \frac{ (\frac{n+2}{n})^{k-j+1}}{q+2\alpha_k}} &\leq \prod_{j=1}^{k} \widetilde{\kappa}^{\frac{ (\frac{n+2}{n})^{k-j+1}}{q+2\alpha_k}} \prod_{j=1}^{k} \beta ^{ \frac{ j(\frac{n+2}{n})^{k-j+1}}{q+2\alpha_k}} \\
    &\leq \prod_{j=1}^{k} \widetilde{\kappa}^{\frac{ (\frac{n+2}{n})^{k-j+1}}{\kappa\left((\frac{n+2}{n}\right)^k-1)}} \prod_{j=1}^{k} \beta ^{ \frac{ j(\frac{n+2}{n})^{k-j+1}}{\kappa((\frac{n+2}{n})^k-1)}}. 
\end{align*}
In order to treat the first quantity in the preceding estimate, we apply the first identity~\eqref{est:algeins} from Lemma~\ref{lem:algebraic} with the data~$(\widetilde{\kappa},\frac{n+2}{n},\kappa)$, which yields
\begin{align*}
    \prod_{j=1}^{k} \widetilde{\kappa}^{\frac{ (\frac{n+2}{n})^{k-j+1}}{\kappa ((\frac{n+2}{n})^k-1)}} = \widetilde{\kappa} ^{\frac{n+2}{2\kappa}}. 
\end{align*}
The second quantity is treated according to the second assertion~\eqref{est:algzwei} from the same lemma, which we apply now with~$(\beta,\frac{n+2}{n},\kappa)$, to obtain
\begin{align*}
    \prod_{j=1}^{k} \beta ^{ \frac{ j(\frac{n+2}{n})^{k-j+1}}{\kappa((\frac{n+2}{n})^k-1)}} \leq \beta ^{\frac{(n+2)^2}{4\kappa}}.
\end{align*}
Recalling our abbreviations for~$\kappa$ and~$\widetilde{\kappa}$, we achieve the following inequality
\begin{align*}
    \prod_{j=1}^{k} \bigg(  4^{j}  2\delta \rho^2 \frac{(1+\alpha_{j})^{3\frac{n+2+\sigma}{\sigma}}}{(s-r)^2} \bigg)^{ \frac{ (\frac{n+2}{n})^{k-j+1}}{q+2\alpha_k}} \leq (\beta\widetilde{\kappa}^{\frac{n+2}{n}})^{\frac{n+2}{4+n(p-q)}}.
\end{align*}
Since both~$\beta$ and~$\widetilde{\kappa}$ are independent of~$k\in\N$, we may pass to the limit~$k\to\infty$ in~\eqref{est:iterationvier}, which yields
\begin{align*}
    \limsup\limits_{k\to\infty} A^{\frac{1}{q+2\alpha_k}}_k &\leq (\beta\widetilde{\kappa}^{\frac{n+2}{n}})^{\frac{n+2}{4+n(p-q)}} \bigg( \fiint_{\Omega_T\cap \Q_s} ( \G(Du)^{q} + 1 )\,\dx\dt \bigg) ^{\frac{2}{4+n(p-q)}} \\
    &\leq C \Bigg( \frac{\rho^{n+2} \Big(1+[[f]]^{\frac{(n+2)(n+2+\sigma)}{\sigma}}_{n+2+\sigma;\Omega_T\cap Q_{2\rho}(z_0)}  \Big) }{(s-r)^{n+2}} \fiint_{\Omega_T\cap \Q_s} (1+|Du|^2)^{\frac{q}{2}} \,\dx\dt \Bigg) ^{\frac{2}{4+n(p-q)}}
\end{align*}
with~$C=C(C_1,C_2,N,n,p,q,\sigma,\Theta_{\rho}(x_0))$. As a last step, we note that~$\alpha_k\to\infty$ as~$k\to\infty$ and also~$\rho_k \to r$ as~$k\to\infty$, which yields the quantitative local~$L^\infty$-gradient estimate 
\begin{align} \label{est:iterationfünf}
    &\sup\limits_{\Omega_T \cap Q_r(z_0)} |Du| \\
    &\qquad\,\,= \lim\limits_{k\to\infty} \bigg( \fiint_{\Omega_T\cap Q_r(z_0)} |Du|^{q+2\alpha_k} \,\dx\dt \bigg) ^{\frac{1}{q+2\alpha_k}} \nonumber \\
    &\qquad\,\,\leq \limsup\limits_{k\to\infty} A^{\frac{1}{q+2\alpha_k}}_k \nonumber \\
    &\qquad\,\,\leq C \bigg( \frac{\rho^{n+2} \Big(1+[[f]]^{\frac{(n+2)(n+2+\sigma)}{\sigma}}_{n+2+\sigma;\Omega_T\cap Q_{2\rho}(z_0)} \Big) }{(s-r)^{n+2}} \fiint_{\Omega_T\cap Q_s(z_0)} (1+|Du|^2)^{\frac{q}{2}} \,\dx\dt \bigg) ^{\frac{1}{q} \cdot \frac{2q}{4+n(p-q)}} \nonumber
\end{align}
with a constant~$C=C(C_1,C_2,N,n,p,q,\sigma,\Theta_{\rho}(x_0))$, which holds true for any~$\rho\leq r<s\leq 2\rho$.


\subsection{Interpolation} \label{subsec:interpolation}
We aim to reduce the integrability exponent~$q\geq p$ present in the right-hand side of the quantitative~$L^\infty$-gradient estimate~\eqref{est:iterationfünf} to the smaller exponent~$p\geq 2$. In order to achieve this, we rewrite the exponent, bound one of the appearing factors from above by taking the essential supremum, and then apply Young's inequality in order to reabsorb the sup-term by aid of the classical iteration Lemma~\ref{lem:iteration}. For this matter, we abbreviate
$$ \alpha = \frac{2q}{4+n(p-q)} $$
and note that due to the assumption~\eqref{parameter} there holds~$\frac{q}{\alpha(q-p)}>1$. Thus, we are allowed to apply Young's inequality with exponents~$(\frac{q}{\alpha(q-p)},\frac{q}{q-\alpha(q-p)})$ and obtain 
\begin{align} \label{quantestpnorm}
     &\sup\limits_{\Omega_T \cap Q_r(z_0)}  (1+|Du|^2)^{\frac{1}{2}} \\
    &\qquad\,\,\leq C \bigg( \frac{\rho^{n+2} \Big(1+[[f]]^{\frac{(n+2)(n+2+\sigma)}{\sigma}}_{n+2+\sigma;\Omega_T\cap Q_{2\rho}(z_0)} \Big) }{(s-r)^{n+2}} \fiint_{\Omega_T\cap Q_s(z_0)} (1+|Du|^2)^{\frac{p}{2} + \frac{q-p}{2}} \,\dx\dt \bigg) ^{\frac{\alpha}{q}} \nonumber \\
    &\qquad\,\,\leq \sup\limits_{\Omega_T \cap Q_s(z_0)}  \Big((1+|Du|^2)^{\frac{1}{2}}\Big)^{\frac{\alpha(q-p)}{q}} \nonumber \\
    &\qquad\,\,\,\quad \cdot C \bigg( \frac{\rho^{n+2} \Big(1+[[f]]^{\frac{(n+2)(n+2+\sigma)}{\sigma}}_{n+2+\sigma;\Omega_T\cap Q_{2\rho}(z_0)} \Big) }{(s-r)^{n+2}} \fiint_{\Omega_T\cap Q_s(z_0)} (1+|Du|^2)^{\frac{p}{2}} \,\dx\dt \bigg) ^{\frac{\alpha}{q}} \nonumber \\
    & \qquad\,\,\leq \frac{1}{2} \sup\limits_{\Omega_T \cap Q_s(z_0)}  (1+|Du|^2)^{\frac{1}{2}} \nonumber \\
    &\qquad\,\,\,\quad + C \bigg( \frac{\rho^{n+2} \Big( 1+[[f]]^{\frac{(n+2)(n+2+\sigma)}{\sigma}}_{n+2+\sigma;\Omega_T\cap Q_{2\rho}(z_0)} \Big) }{(s-r)^{n+2}} \fiint_{\Omega_T\cap Q_s(z_0)} (1+|Du|^2)^{\frac{p}{2}} \,\dx\dt \bigg) ^{\frac{\alpha}{q-\alpha(q-p)}}. \nonumber
\end{align}
 The claimed estimate~\eqref{est:apriorigradientsuper} now follows by an application of the standard iteration Lemma~\ref{lem:iteration} and the definition of~$[[f]]_{n+2+\sigma;\Omega_T\cap Q_{2\rho}(z_0)}$ according to~\eqref{est:fabkürz}. 


\section{Approximation} \label{sec:approx}

Since the \textit{a priori} estimate~\eqref{est:apriorigradientsuper} from Section~\ref{sec:aprioriest} is stated for classical solutions~$u\in C^3(\overline{\Omega}_T\cap Q_{2\rho}(z_0),\R^N)$, we now perform a regularizing procedure of the original parabolic system~\eqref{pde}. This approach yields an approximating family of solutions that are regular up to the boundary, allowing us to exploit the \textit{a priori} gradient estimate from Proposition~\ref{prop:apriorigradientsuper}. Throughout this section, let
$$ u\in C^0\big(\Lambda_{2\rho}(t_0);L^2(\Omega \cap B_{2\rho}(x_0),\R^N) \big) \cap L^q \big(\Lambda_{2\rho}(t_0);W^{1,q}(\Omega \cap B_{2\rho}(x_0),\R^N)\big) $$
  denote a local weak solution to the parabolic system~\eqref{pde} under the set of assumptions~\eqref{parameter} -- \eqref{datumregularität}, which satisfies~$u\equiv 0$ on~$(\partial\Omega)_T \cap Q_{2\rho}(z_0)$ in the sense of traces for some~$z_0=(x_0,t_0)$ with spatial boundary point~$x_0\in\partial\Omega$ and~$\Lambda_{2\rho}(t_0)\subset(0,T).$ 


\subsection{Approximation of the coefficients} \label{subsec:approxvectorfield}

We follow the approach from~\cite[Chapter 4.2]{boundaryparabolic}. Let us consider a standard positive and smooth mollifier $\phi\in C^\infty_0((-1,1))$ that satisfies $\int_{\R} \phi\,\dx =1$. For $\epsilon\in(0,1]$, we further denote its scaled variant by $\phi_\epsilon(x) \coloneqq \frac{\phi(\frac{x}{\epsilon})}{\epsilon}$ that is compactly supported within~$(-\epsilon,\epsilon)$. Moreover, we define
$$ V(s) \coloneqq a(e^s) $$
and let~$V_\epsilon(s) \coloneqq (V* \phi_\epsilon)(s) $ for any~$s\in\R$. The regularized coefficients are then given by
\begin{align} \label{approx}
    a_\epsilon(s) \coloneqq V_\epsilon(\log(s))
\end{align}
for~$s>0$. As outlined in~\cite[Appendix~A]{boundaryparabolic} and~\cite[Section~4.2]{boundaryellipticpq}, these regularized coefficients satisfy the following set of~$(p,q)$-growth conditions
\begin{align} \label{approxgrowth}
  \begin{cases}
       \widetilde{C}_1 (\mu^2+s^2)^{\frac{p-2}{2}} \leq a_\epsilon(s) \leq \widetilde{C}_2 (\mu^2+s^2)^{\frac{q-2}{2}}, \\
       \widetilde{C}_1 (\mu^2+s^2)^{\frac{p-2}{2}} \leq  a_\epsilon(s) + a'_\epsilon(s)s \leq \widetilde{C}_2 (\mu^2+s^2)^{\frac{q-2}{2}}, \\
       |a''_\epsilon(s)s^2| \leq \frac{\widetilde{C}_2}{\epsilon} (\mu^2+s^2)^{\frac{q-2}{2}} 
  \end{cases} 
\end{align}
with positive constants~$\widetilde{C}_1=\widetilde{C}_1(C_1,n,p)$ and~$\widetilde{C}_2=\widetilde{C}_2(C_2,n,q)$. Additionally, for any~$\epsilon\in(0,1]$ there holds the quantitative estimate
\begin{align} \label{smallnessest}
    |a_\epsilon(s) - a(s)| \leq  \epsilon \widetilde{C}(C_2,q) e^{q-2}(\mu^2+s^2)^{\frac{q-2}{2}}
\end{align}
for any~$s>0$. Now, we consider
\begin{equation} \label{coeffqgrowth}
    b_\epsilon(s) \coloneqq \epsilon(\mu^2+s^2)^{\frac{q-2}{2}} + a_\epsilon(s)
\end{equation}
for~$s>0$. According to the structure conditions~\eqref{approxgrowth}, we infer that the coefficients~$b_\epsilon$ satisfy the following~$q$-growth conditions
\begin{align} \label{approxqgrowth}
  \begin{cases}
       \epsilon(\mu^2+s^2)^{\frac{q-2}{2}}  \leq b_\epsilon(s) \leq \widehat{C}_2 (\mu^2+s^2)^{\frac{q-2}{2}}, \\
       \epsilon\widehat{C}_1 (\mu^2+s^2)^{\frac{q-2}{2}} \leq  b_\epsilon(s) + b'_\epsilon(s)s \leq \widehat{C}_2 (\mu^2+s^2)^{\frac{q-2}{2}}, \\
       |b''_\epsilon(s)s^2| \leq \widehat{C}_2(\frac{1}{\epsilon}+1) (\mu^2+s^2)^{\frac{q-2}{2}}
  \end{cases} 
\end{align}
with positive constants~$\widehat{C}_1=\widehat{C}_1(C_1,n,q)$ and~$\widehat{C}_2=\widehat{C}_2(C_2,n,q)$. We remark that by discarding the contribution of the coefficient's first term in~\eqref{coeffqgrowth}, it is easy to check that~$b_\epsilon$ also satisfies nonstandard~$(p,q)$-growth conditions of the form~\eqref{approxgrowth} with structure constants independent of~$\epsilon\in(0,1]$. 


\subsection{Approximation of the system} \label{subsec:approxequation}

The next step of our regularizing procedure lies in utilizing the previous steps outlined in Section~\ref{subsec:approxvectorfield}. As in the statement of Theorem~\ref{hauptresultat}, we assume that~$u$ vanishes on the lateral boundary of~$(\partial\Omega)_T \cap Q_{2\rho}(z_0)$, i.e. that there holds
$$u \equiv 0 \qquad \mbox{on~$(\partial\Omega)_T \cap Q_{2\rho}(z_0)$}$$ 
in the sense of traces, where~$\Lambda_{2\rho}(t_0)\subset(0,T)$. The datum~$f\in L^{n+2+\sigma}(\Omega_T,\R^n)$ is first localized by considering its truncation~$f_E \coloneqq f \bigchi_E$, where we set
$$ E \coloneqq \{x\in\Omega:\dist(x,\partial\Omega)>\epsilon\}\times(0,T) \subset \Omega_T $$
for~$\epsilon\in(0,1]$ small enough. Then, we regularize~$f_E$ via~$f_\epsilon \coloneqq f_E*\eta_{\frac{\epsilon}{2}}$, where~$\eta\in C^\infty_0(B_1\times(-1,1))$ denotes a standard mollifier and~$\eta_\epsilon(z)\coloneqq\epsilon^{-1} \eta(\frac{z}{\epsilon})$. Accordingly, we have that~$f_\epsilon\in C^1(\overline{\Omega}_T\cap Q_{2\rho}(z_0),\R^N)$ and~$\supp f_\epsilon \Subset \Omega\times\R$, ensuring that the assumptions of the \textit{a priori} gradient estimate from Proposition~\ref{prop:apriorigradientsuper} are satisfied by the inhomogeneity~$f_\epsilon$. We now consider the unique weak solution
$$ u_\epsilon \in L^q\big( \Lambda_{2\rho}(t_0); W^{1,q}(\Omega\cap B_{2\rho}(x_0), \R^N) \big) $$
to the Cauchy-Dirichlet problem
\begin{align} \label{cauchydirichlet}
    \begin{cases}
        \partial_t u^i_\epsilon - \divv \big( b_{\epsilon}(|Du_\epsilon|) Du^i_\epsilon \big)= f^i_\epsilon  & \mbox{in~$\Omega_T \cap Q_{2\rho}(z_0)$}, \\
         u^i_\epsilon = u & \mbox{on~$\partial_p(\Omega_T \cap Q_{2\rho}(z_0))$} 
    \end{cases}
\end{align}
for any $i=1,\ldots,N\in\N$ and~$\epsilon\in(0,1]$, where the coefficients~$(b_{\epsilon})_{\epsilon\in (0,1]}$ are defined in~\eqref{coeffqgrowth}.
In particular, there holds~$u_\epsilon \equiv 0 $ on the lateral boundary~$(\partial\Omega)_T\cap Q_{2\rho}(z_0)$ due to the assumptions made on the local weak solution~$u$. Here, the initial datum~$u(\cdot,t_0-4\rho^2) \in L^2(\Omega_T \cap Q_{2\rho}(z_0),\R^N)$ is taken in the usual~$L^2$-sense, i.e.
\begin{equation} \label{randwertel2}
    \lim\limits_{h\downarrow 0} \fint_{t_0-4\rho^2}^{t_0-4\rho^2+h} \int_{\Omega_T \cap Q_{2\rho}(t_0)} |u_\epsilon(x,t)- u(x,t_0-4\rho^2)|^2 \,\dx\dt =0. 
\end{equation}
 Indeed, this approach is feasible according to our notion of local weak solutions in~\ref{defweakform}. The existence of a unique weak solution~$u_\epsilon$ now follows from~\cite[\S 2.1, Th{\'e}or{\'e}me 1.2 and Remarque 1.10]{lions1969quelques}, which asserts that
$$ u_\epsilon\in C^0\big(\Lambda_{2\rho}(t_0);L^2(\Omega\cap B_{2\rho}(x_0),\R^N) \big) \cap L^q\big(\Lambda_{2\rho}(t_0);u+W^{1,q}_0(\Omega\cap B_{2\rho}(x_0),\R^N) \big) $$
with a time derivative~$\partial_t u_\epsilon$ in space
$$ \partial_t u_\epsilon \in L^{\frac{q}{q-1}}\big( \Lambda_{2\rho}(t_0);W^{-1,\frac{q}{q-1}}(\Omega\cap B_{2\rho}(x_0),\R^N) \big). $$
In particular, for any~$\phi \in L^q\big(\Lambda_{2\rho}(t_0);u+W^{1,q}_0(\Omega\cap B_{2\rho}(x_0),\R^N) \big)$ there holds the weak form
\begin{align} \label{weakformduality}
    \iint_{\Omega_T\cap Q_{2\rho}(z_0)} \langle\partial_t u_\epsilon,u_\epsilon - \phi \rangle_{W^{1,q}_0(\Omega\cap B_{2\rho}(x_0),\R^N)} \, \dx\dt & + \iint_{\Omega_T\cap Q_{2\rho}(z_0)} b_\epsilon(|Du_\epsilon|)Du_\epsilon \cdot(Du_\epsilon-D\phi) \,\dx \dt \\
    & = \iint_{\Omega_T\cap Q_{2\rho}(z_0)} f_\epsilon \cdot(u_\epsilon-\phi) \, \dx\dt, \nonumber
\end{align}
where~$\langle\cdot,\cdot\rangle$ denotes the duality pairing between the spaces~$W^{1,q}_0$ and~$W^{-1,\frac{q}{q-1}}$. \,\\ 

At this point, we note that the coefficients~$(b_{\epsilon})_{\epsilon\in(0,1]}$ are smooth for any~$\epsilon\in (0,1]$ by construction and, according to~\eqref{approxqgrowth}, satisfy a standard~$q$-growth condition. In fact, the quantitative constants in~\eqref{approxqgrowth} depend on the parameter~$\epsilon\in(0,1]$ in general. However, the following regularity result for~$u_\epsilon$ is only used in a qualitative way and the presence of the approximating parameter~$\epsilon\in(0,1]$ within the structure constants does not pose any problem. Moreover,~$\Omega\subset\R^n$ is assumed to be a bounded and convex~$C^2$-set. Consequently, we are in the setting of~\cite[Appendix B: Regularity up to the boundary]{boundaryparabolic}, where we recall the standard~$q$-growth conditions~\eqref{approxqgrowth} for~$b_\epsilon$. By following their approach, which builds upon global Schauder estimates and a reflection technique, in combination with local~$C^{1,\alpha}$-regularity results, we obtain that the approximating solutions~$u_\epsilon$ are regular up to the boundary, i.e. there holds
\begin{align} \label{approxsolsmooth}
    u_\epsilon \in C^3\big( \overline{\Omega}_T \cap Q_{2\rho}(z_0),\R^N\big)
\end{align}
for any~$\epsilon\in (0,1]$. In other words,~$u_\epsilon$ is a regular, classical solution to~\eqref{cauchydirichlet} for any~$\epsilon\in(0,1]$. We also mention similar approximation procedures outlined by Bögelein, Duzaar, Marcellini, and Scheven in~\cite[Section~5]{boundaryellipticpq} and also by De Filippis and Piccinini in~\cite[Section~5.3]{cristiana}. Our next aim is the elaboration of appropriate energy estimates involving the approximating mappings~$(u_\epsilon)_{\epsilon\in(0,1]}$. 


\subsection{Comparison estimate} \label{subsec:energyestapproxsol}

In the following, we prove a comparison estimate involving the approximating weak solutions~$u_\epsilon$ and the original local weak solution~$u$. 


\begin{mylem} \label{lem:energyestimate}

For any~$\epsilon\in(0,1]$ there holds the energy estimate
\begin{align} \label{est:energyestimate}
  \sup\limits_{t\in\Lambda_{2\rho}(t_0)} &\int_{(\Omega\cap B_{2\rho}(x_0))\times\{t\}} |u_\epsilon - u|^2\,\dx + \iint_{\Omega_T \cap Q_{2\rho}(z_0) } |Du_\epsilon - Du|^{p}\,\dx\dt \\
     & \leq C \epsilon \iint_{\Omega_T \cap Q_{2\rho}(z_0)} (\mu^2+|Du|^2)^{\frac{q-2}{2}}|Du|^2\,\dx\dt \nonumber \\
     & \quad + C|\Omega_T\cap Q_{2\rho}(z_0)|^{\frac{p}{p-1}} \iint_{\Omega_T \cap Q_{2\rho}(z_0)} |f_\epsilon-f|^{\frac{p}{p-1}}\,\dx\dt \nonumber
\end{align}
with~$C=C(C_1,n,p)$. 
\end{mylem}


\begin{proof}
We begin by subtracting the weak form for~$u_\epsilon$ and~$u$ on~$\Omega_T \cap Q_{2\rho}(z_0)$ and test the resulting equations with the test function~$\phi= \eta^2_\gamma (u_\epsilon-u) $. According to our construction of~$u_\epsilon$ in~\eqref{cauchydirichlet}, this choice of test function is indeed admissible up to an approximation argument. Here,~$\eta_\gamma\in W^{1,\infty}(\R)$ denotes a cut-off function in time, given by
\begin{align*}
    \eta_\gamma(t) = \begin{cases}
        1 & \mbox{for~$t\in(-\infty,\tau)$}, \\
        1 - \frac{t-\tau}{\gamma} & \mbox{for~$t\in[\tau,\tau+\gamma)$}, \\
        0 & \mbox{for~$t\in[\tau+\gamma,t_0)$},
    \end{cases}
\end{align*} 
for an arbitrary but fixed time~$\tau\in \Lambda_{2\rho}(t_0)\setminus\{t_0\}$ and a small parameter~$\gamma \in (0,\min\{ \tau-(t_0-4\rho^2), t_0-\tau\})$. At this point, we recall the identity~\eqref{weakformduality} and note that the local weak solution~$u$ exhibits a time derivative in the very same dual space, which can be inferred in a straightforward manner by exploiting the growth condition~\eqref{A-Voraussetzung2} and Hölder's inequality, cf.~\cite[Section 6.6]{bogelein2013parabolic}. Plugging the above test function into the resulting weak formulation involving~$u_\epsilon$ and~$u$ respectively, and integrating by parts, there holds 
\begin{align} \label{est:energyestweakform}
    \iint_{\Omega_T\cap Q_{2\rho}(z_0)} & \langle \partial_t  (u_\epsilon-u), \eta^2_\gamma ( u_\epsilon-u) \rangle_{W^{1,q}_0(\Omega\cap B_{2\rho}(x_0),\R^N)} \,\dx\dt \,\\
    &\quad + \iint_{\Omega_T\cap Q_{2\rho}(z_0)}  \eta^2_\gamma (b_{\epsilon}(|Du_\epsilon|) Du_\epsilon - a(|Du|)Du) \cdot (Du_\epsilon-Du) \,\dx\dt \nonumber \\
    &=  \iint_{\Omega_T\cap Q_{2\rho}(z_0)}  \eta^2_\gamma (f_\epsilon-f) \cdot (u_\epsilon-u)\,\dx\dt. \nonumber
\end{align}
The term involving the time derivative can be rewritten by the product rule as
\begin{align*}
     \iint_{\Omega_\epsilon\times\Lambda_{2\rho}(t_0)} &  \langle \partial_t  (u_\epsilon-u),\eta^2_\gamma( u_\epsilon-u ) \rangle_{W^{1,q}_0(\Omega\cap B_{2\rho}(x_0),\R^N)} \,\dx\dt \\
     &= \iint_{\Omega_T\cap Q_{2\rho}(z_0)} \langle (\partial_t (\eta_\gamma(u_\epsilon-u)), (\eta_\gamma (u_\epsilon-u))) \rangle_{W^{1,q}_0(\Omega\cap B_{2\rho}(x_0),\R^N)} \,\dx\dt \\
     &\quad - \iint_{\Omega_T\cap Q_{2\rho}(z_0)} \eta_\gamma' \eta_\gamma |u_\epsilon-u|^2 \,\dx\dt \\
    &= \frac{1}{2} \iint_{\Omega_T\cap Q_{2\rho}(z_0)} \partial_t |\eta_\gamma (u_\epsilon-u)|^2\,\dx\dt 
 - \iint_{\Omega_T\cap Q_{2\rho}(z_0)} \eta_\gamma' \eta_\gamma |u_\epsilon-u|^2 \,\dx\dt  \\
    &\eqqcolon \foo{I}_\gamma + \foo{II}_\gamma.
\end{align*}
The first quantity~$\foo{I}_\gamma$ vanishes, since we have~$u_\epsilon=u$ at the initial time~$t=t_0-4\rho^2$ and also~$\eta_\gamma =0$ at~$t=t_0$. On the contrary, by passing to the limit~$\gamma\downarrow 0$ in the second term, we obtain
\begin{align*} 
    \lim\limits_{\gamma\downarrow 0} \foo{II}_\gamma &= \int_{(\Omega\cap B_{2\rho}(x_0)) \times\{\tau\}}  |u_\epsilon-u|^2 \,\dx.  
\end{align*}
  In the remaining two integral quantities in~\eqref{est:energyestweakform}, we straightaway pass to the limit~$\gamma\downarrow 0$, which leads us to the identity
  \begin{align} \label{est:energylimit}
    & \int_{ (\Omega \cap B_{2\rho}(x_0)) \times\{\tau\}}  |u_\epsilon - u|^2 \,\dx \\
    & \qquad \quad + \iint_{\Omega_\tau \cap Q_{2\rho}(z_0)} (b_{\epsilon}(|Du_\epsilon|) Du_\epsilon - a(|Du|)Du) \cdot (Du_\epsilon-Du) \,\dx\dt \nonumber \\
    & \quad = \iint_{\Omega_\tau \cap Q_{2\rho}(z_0)} (f_\epsilon-f) \cdot (u_\epsilon-u)\,\dx\dt \nonumber
\end{align} 
  for any~$\tau\in \Lambda_{2\rho}(t_0)\setminus\{t_0\}$. Next, we concentrate on the treatment of the second term. For this, we use the definition of the coefficients~$b_\epsilon$ and rewrite the quantity as
\begin{align*}
    (b_{\epsilon}&(|Du_\epsilon|) Du_\epsilon - a(|Du|)Du) \cdot (Du_\epsilon-Du) \\
    & = (a_\epsilon(|Du_\epsilon|)Du_\epsilon - a_\epsilon(|Du|)Du) \cdot (Du_\epsilon-Du) \\
    &\quad + (a_\epsilon(|Du|Du-a(|Du|)Du) \cdot (Du_\epsilon - Du) \\
    &\quad + \epsilon (\mu^2+|Du_\epsilon|^2)^{\frac{q-2}{2}}|Du_\epsilon|^2 - \epsilon (\mu^2+|Du_\epsilon|^2)^{\frac{q-2}{2}} Du_\epsilon \cdot Du.
\end{align*}
In order to treat the first term, we employ the ellipticity estimate~\eqref{est:ellipticity} from Lemma~\ref{lem:ellipticity} to achieve
\begin{align*}
    (a_\epsilon(|Du_\epsilon|)Du_\epsilon & - a_\epsilon(|Du|)Du) \cdot (Du_\epsilon-Du)\\
    &\geq C \int_{0}^{1} (\mu^2+|Du+s(Du_\epsilon-Du)|^2)^{\frac{p-2}{2}}\, \ds |Du_\epsilon-Du|^2 \\
    & \geq C (\mu^2+|Du|^2+|Du_\epsilon|^2)^{\frac{p-2}{2}}|Du_\epsilon-Du|^p \\
    &\geq C |Du_\epsilon-Du|^p
\end{align*}
with a constant~$C=C(C_1,p)$. The second term is bounded from below by distinguishing between the cases where~$|Du|\leq \frac{1}{2}|Du_\epsilon|$ and~$|Du|>\frac{1}{2}|Du_\epsilon|$. Indeed, in the former case, we simply obtain
\begin{align*}
    - \epsilon (\mu^2 & +|Du_\epsilon|^2)^{\frac{q-2}{2}} Du_\epsilon \cdot Du \\
    &\geq - \epsilon (\mu^2+|Du_\epsilon|^2)^{\frac{q-2}{2}} |Du_\epsilon| |Du| \\
    &\geq - \mbox{$\frac{1}{2}$} \epsilon (\mu^2+|Du_\epsilon|^2)^{\frac{q-2}{2}} |Du_\epsilon|^2,
\end{align*}
whereas in the other case we apply Young's inequality to achieve
\begin{align*}
    - \epsilon (\mu^2 & +|Du_\epsilon|^2)^{\frac{q-2}{2}} Du_\epsilon \cdot Du \\
    &\geq - \epsilon (\mu^2+|Du_\epsilon|^2)^{\frac{q-2}{2}} |Du_\epsilon| |Du| \\
    &\geq - \mbox{$\frac{1}{2}$} \epsilon (\mu^2+|Du_\epsilon|^2)^{\frac{q-2}{2}} |Du_\epsilon|^2 - \mbox{$\frac{1}{2^{1-q}}$} \epsilon (\mu^2+|Du|^2)^{\frac{q-2}{2}} |Du|^2. 
\end{align*}
Consequently, the second quantity is bounded from below by
\begin{align*}
    (a_\epsilon(|Du|Du&-a(|Du|)Du) \cdot (Du_\epsilon - Du) \\
    &\geq - \mbox{$\frac{1}{2}$} \epsilon (\mu^2+|Du_\epsilon|^2)^{\frac{q-2}{2}} |Du_\epsilon|^2 - \mbox{$\frac{1}{2^{1-q}}$} \epsilon (\mu^2+|Du|^2)^{\frac{q-2}{2}} |Du|^2. 
\end{align*}
Finally, for the third term we use estimate~\eqref{smallnessest} and argue similarly as for the second term by distinguishing between the cases where either~$|Du|\leq \mbox{$\frac{1}{2}$}|Du_\epsilon|$ or~$|Du|> \mbox{$\frac{1}{2}$}|Du_\epsilon|$ holds true. Young's inequality then yields the lower bound
\begin{align*} 
    \epsilon (\mu^2+|Du_\epsilon|^2)^{\frac{q-2}{2}}|Du_\epsilon|^2 & - \epsilon (\mu^2+|Du_\epsilon|^2)^{\frac{q-2}{2}} Du_\epsilon \cdot Du \\
    & \geq - \mbox{$\frac{1}{2}$} \epsilon (\mu^2+|Du_\epsilon|^2)^{\frac{q-2}{2}} |Du_\epsilon|^2 - C \epsilon (\mu^2+|Du|^2)^{\frac{q-2}{2}} |Du|^2
\end{align*} 
with a constant~$C=C(C_2,q)$. Combining the preceding estimates for all three quantities, we obtain the lower bound
\begin{align*}
    (b_{\epsilon}&(|Du_\epsilon|) Du_\epsilon - a(|Du|)Du) \cdot (Du_\epsilon-Du) \geq c |Du_\epsilon-Du|^p - C \epsilon (\mu^2+|Du|^2)^{\frac{q-2}{2}} |Du|^2 
\end{align*}
with constants~$c=c(C_1,p)$ and~$C=C(C_2,q)$ that remain independent of~$\epsilon\in(0,1]$. Eventually, the term on the right-hand side involving the datum~$f$ in~\eqref{est:energylimit} is bounded above by Hölder's and Young's inequality and also by a slice-wise application of Poincaré's inequality. Notice that this is possible due the assumption~$u_\epsilon=u$ on~$\partial_p(\Omega_T\cap Q_{2\rho}(z_0))$. Moreover, we exploit the fact that there holds~$\frac{p}{p-1}<n+2+\sigma$. This way, we obtain 
\begin{align*}
    &\iint_{\Omega_\tau\cap Q_{2\rho}(z_0)} (f_\epsilon-f) \cdot (u_\epsilon-u)\,\dx\dt \\
    &~~\leq \bigg(\iint_{\Omega_\tau\cap Q_{2\rho}(z_0)}|f_\epsilon-f|^{\frac{p}{p-1}}\,\dx\dt \bigg)^{\frac{p-1}{p}} \bigg(\iint_{\Omega_\tau\cap Q_{2\rho}(z_0)}|u_\epsilon-u|^{p}\,\dx\dt \bigg)^{\frac{1}{p}} \\
    &~~\leq C|\Omega_T\cap Q_{2\rho}(z_0)| \bigg(\iint_{\Omega_\tau\cap Q_{2\rho}(z_0)}|f_\epsilon-f|^{\frac{p}{p-1}}\,\dx\dt \bigg)^{\frac{p-1}{p}} \bigg(\iint_{\Omega_\tau\cap Q_{2\rho}(z_0)}|Du_\epsilon-Du|^{p}\,\dx\dt \bigg)^{\frac{1}{p}} \\
    &~~\leq \mbox{$\frac{1}{2}$} c \iint_{\Omega_\tau\cap Q_{2\rho}(z_0)}|Du_\epsilon-Du|^{p}\,\dx\dt + C |\Omega_T\cap Q_{2\rho}(z_0)|^{\frac{p}{p-1}} \iint_{\Omega_\tau\cap Q_{2\rho}(z_0)}|f_\epsilon-f|^{\frac{p}{p-1}}\,\dx\dt
\end{align*}
with a constant~$C=C(C_1,n,p)$. Combining all our estimates and taking the supremum over~$\tau\in\Lambda_{2\rho}(t_0)\setminus\{t_0\}$, the claimed comparison estimate~\eqref{est:energyestimate} follows. 
\end{proof}


\section{Proof of the main result} \label{sec:hauptresultatbeweis}

In this section, we give the proof of our main result, i.e. of Theorem~\ref{hauptresultat}. We recall that the approximating solutions~$u_\epsilon$ are regular up to the boundary and in particular satisfy
$$ u_\epsilon \in C^3\big(\overline{\Omega}_T \cap Q_{2\rho}(z_0),\R^N\big) $$
for any~$\epsilon\in(0,1]$, as stated in~\eqref{approxsolsmooth} within Section~\ref{subsec:approxequation}. Moreover, according to the reasoning provided in Section~\ref{subsec:approxequation}, the mollified datum~$f_\epsilon$ is smooth by construction for any~$\epsilon\in(0,1]$ and satisfies~$\supp f_\epsilon \Subset \Omega\times\R$. Thus, the~\textit{a priori}~$L^\infty$-gradient estimate~\eqref{est:apriorigradientsuper} from Proposition~\ref{prop:apriorigradientsuper} is at our disposal for~$u_\epsilon$, which we apply to the data~$(b_\epsilon$,~$f_\epsilon$,~$\Omega)$. The quantitative estimate~\eqref{est:apriorigradientsuper} reads as
\begin{align} \label{est:mainresulteins}
  & \sup\limits_{(\Omega_\epsilon)_T \cap Q_\rho(z_0)} |Du_\epsilon| \\
  & \,\,\,\quad \leq C\bigg( \Big( \|f_\epsilon\|^{\frac{(n+2)(n+2+\sigma)}{\sigma}}_{L^{n+2+\sigma}((\Omega_\epsilon)_T\cap Q_{2\rho}(z_0))} \rho^{n+2} + 1 \Big)  \fiint_{(\Omega_\epsilon)_T \cap Q_{2\rho}(z_0)} (1+|Du_\epsilon|^{p}) \,\dx\dt \bigg) ^{\frac{2}{4+(p-q)(n+2)}} \nonumber
\end{align}
with~$C=C(C_1,C_2,N,n,p,q,\sigma,\Theta_{\rho}(x_0))\geq 1$. We note that the constant~$C$ in particular does not depend on~$\epsilon\in(0,1]$, which is an immediate consequence of the comment made right after~\eqref{approxqgrowth}, asserting that~$b_\epsilon$ satisfies a nonstandard~$(p,q)$-growth condition of the type~\eqref{approxgrowth}. From the comparison estimate~\eqref{est:energyestimate} from Lemma~\ref{lem:energyestimate} we infer that~$Du_\epsilon \to Du$ strongly in~$L^p(\Omega_T\cap Q_{2\rho}(z_0),\R^N)$ as~$\downarrow 0$, since there holds~$f_\epsilon \to f$ in~$L^{\frac{p}{p-1}}(\Omega_T\cap Q_{2\rho}(z_0),\R^N)$ as~$\downarrow 0$. Moreover, we may pass to a subsequence~$(\epsilon_i)_{i\in\N}$ to infer that~$Du_{\epsilon_i} \to Du$ a.e. in~$\Omega_T \cap Q_{2\rho}(z_0)$. For the datum~$f_\epsilon$, we furthermore have that~$f_\epsilon \to f$ in~$L^{n+2+\sigma}(\Omega_T\cap Q_{2\rho}(z_0),\R^N)$ in the limit~$\epsilon \downarrow 0$. Consequently, we may pass to the limit~$\epsilon_i \downarrow 0$ on both sides of~\eqref{est:mainresulteins} and use the lower semicontinuity of the~$L^\infty$-norm with respect to a.e. convergence, to deduce that
\begin{align*} 
  &\sup\limits_{\Omega_T \cap Q_{\rho}(z_0)} |Du| \\
  &\qquad\,\,\leq C\bigg( \Big( \|f\|^{\frac{(n+2)(n+2+\sigma)}{\sigma}}_{L^{n+2+\sigma}(\Omega_T\cap Q_{2\rho}(z_0))} \rho^{n+2} + 1 \Big)  \fiint_{\Omega_T \cap Q_{2\rho}(z_0)} (1+|Du|^{p}) \,\dx\dt \bigg) ^{\frac{2}{4+(p-q)(n+2)}} \nonumber
\end{align*}
with~$C=C(C_1,C_2,N,n,p,q,\sigma,\Theta_{\rho}(x_0))\geq 1$. This finishes the proof of Theorem~\ref{hauptresultat}. 


\nocite{*}
\bibliographystyle{plain}
\bibliography{Boundary_Regularity_Strunk}

\begin{thebibliography}{10}

\bibitem{acerbi1989regularity}
E.~Acerbi and N.~Fusco.
\newblock Regularity for minimizers of non-quadratic functionals: the case $1< p< 2$.
\newblock {\em Journal of Mathematical Analysis and Applications}, 140(1):115--135, 1989.

\bibitem{antonini2026local}
C.A. Antonini.
\newblock Local and global {C}$^{1,\beta}$-regularity for uniformly elliptic quasilinear equations of $ p $-{L}aplace and {O}rlicz-{L}aplace type.
\newblock {\em arXiv preprint arXiv:2601.07140}, 2026.

\bibitem{banerjee2014gradient}
A.~Banerjee and J.L. Lewis.
\newblock Gradient bounds for $p$-harmonic systems with vanishing {N}eumann ({D}irichlet) data in a convex domain.
\newblock {\em Nonlinear Analysis: Theory, Methods \& Applications}, 100:78--85, 2014.

\bibitem{bogelein2015global}
V.~B{\"o}gelein.
\newblock Global gradient bounds for the parabolic $p$-{L}aplacian system.
\newblock {\em Proceedings of the London Mathematical Society}, 111(3):633--680, 2015.

\bibitem{boundaryparabolic}
V.~B{\"o}gelein, F.~Duzaar, N.~Liao, and C.~Scheven.
\newblock Boundary regularity for parabolic systems in convex domains.
\newblock {\em Journal of the London Mathematical Society}, 105(3):1702--1751, 2022.

\bibitem{bogelein2013parabolic}
V.~B{\"o}gelein, F.~Duzaar, and P.~Marcellini.
\newblock Parabolic equations with $p, q$-growth.
\newblock {\em Journal de Math{\'e}matiques Pures et Appliqu{\'e}es}, 100(4):535--563, 2013.

\bibitem{boundaryellipticpq}
V.~B{\"o}gelein, F.~Duzaar, P.~Marcellini, and C.~Scheven.
\newblock Boundary regularity for elliptic systems with $p,q$-growth.
\newblock {\em Journal de Math{\'e}matiques Pures et Appliqu{\'e}es}, 159:250--293, 2022.

\bibitem{bogelein2011degenerate}
V.~B{\"o}gelein, F.~Duzaar, and G.~Mingione.
\newblock Degenerate problems with irregular obstacles.
\newblock {\em Journal f{\"u}r die Reine und Angewandte Mathematik}, 2011(650), 2011.

\bibitem{carozza2011higher}
M.~Carozza, J.~Kristensen, and A.~Passarelli~di Napoli.
\newblock Higher differentiability of minimizers of convex variational integrals.
\newblock {\em Annales de l'Institut Henri Poincar{\'e} C, Analyse Non Lin{\'e}aire}, 28(3):395--411, 2011.

\bibitem{carozza2013regularity}
M.~Carozza, J.~Kristensen, and A.~Passarelli~di Napoli.
\newblock Regularity of minimizers of autonomous convex variational integrals.
\newblock {\em arXiv preprint arXiv:1310.4435}, 2013.

\bibitem{cianchi2010global}
A.~Cianchi and V.G. Maz'ya.
\newblock Global {L}ipschitz regularity for a class of quasilinear elliptic equations.
\newblock {\em Communications in Partial Differential Equations}, 36(1):100--133, 2010.

\bibitem{cianchi2014global}
A.~Cianchi and V.G. Maz’ya.
\newblock Global boundedness of the gradient for a class of nonlinear elliptic systems.
\newblock {\em Archive for Rational Mechanics and Analysis}, 212:129--177, 2014.

\bibitem{colombo2015bounded}
M.~Colombo, G.~Mingione, et~al.
\newblock Bounded minimisers of double phase variational integrals.
\newblock {\em Arch. Ration. Mech. Anal}, 218(1):219--273, 2015.

\bibitem{cristiana}
C.~De~Filippis and M.~Piccinini.
\newblock Borderline global regularity for nonuniformly elliptic systems.
\newblock {\em International Mathematics Research Notices}, 2023(20):17324--17376, 2023.

\bibitem{dibenedetto1993degenerate}
E.~DiBenedetto.
\newblock {\em Degenerate parabolic equations}.
\newblock Springer Science \& Business Media, 1993.

\bibitem{dibenedetto1989boundary}
E.~DiBenedetto and Y.Z. Chen.
\newblock Boundary estimates for solutions of nonlinear degenerate parabolic systems.
\newblock {\em Reine Angew. Math.}, 395:102--131, 1989.

\bibitem{friedman1984regularity}
E.~DiBenedetto and A.~Friedman.
\newblock Regularity of solutions of nonlinear degenerate parabolic systems.
\newblock {\em J. Reine Angew. Math.}, 349:83--128, 1984.

\bibitem{dibenedetto1985addendum}
E.~DiBenedetto and A.~Friedman.
\newblock Addendum to “{H}{\"o}lder estimates for nonlinear degenerate parabolic systems”.
\newblock {\em J. Reine Angew. Math.}, 363:217--220, 1985.

\bibitem{Friedman1985}
E.~DiBenedetto and A.~Friedman.
\newblock Hölder estimates for nonlinear degenerate parabolic sytems.
\newblock {\em J. Reine Angew. Math.}, 357:1--22, 1985.

\bibitem{dibenedetto1992note}
E.~DiBenedetto, J.~Manfredi, and V.~Vespri.
\newblock A note on boundary regularity for certain degenerate parabolic equations.
\newblock In {\em Nonlinear Diffusion Equations and Their Equilibrium States, 3: Proceedings from a Conference held August 20--29, 1989 in Gregynog, Wales}, pages 177--182. Springer, 1992.

\bibitem{duzaar2010local}
F.~Duzaar and G.~Mingione.
\newblock Local {L}ipschitz regularity for degenerate elliptic systems.
\newblock {\em Annales de l'Institut Henri Poincar{\'e} C}, 27(6):1361--1396, 2010.

\bibitem{eleuteri2016lipschitz}
M.~Eleuteri, P.~Marcellini, and E.~Mascolo.
\newblock Lipschitz estimates for systems with ellipticity conditions at infinity.
\newblock {\em Annali di Matematica Pura ed Applicata (1923-)}, 195:1575--1603, 2016.

\bibitem{eleuteri2020regularity}
M.~Eleuteri, P.~Marcellini, and E.~Mascolo.
\newblock Regularity for scalar integrals without structure conditions.
\newblock {\em Advances in Calculus of Variations}, 13(3):279--300, 2020.

\bibitem{foss2008global}
M.~Foss.
\newblock Global regularity for almost minimizers of nonconvex variational problems.
\newblock {\em Annali di Matematica Pura ed Applicata (1923-)}, 2(187):263--321, 2008.

\bibitem{giaquinta1986partial}
M.~Giaquinta and G.~Modica.
\newblock Partial regularity of minimizers of quasiconvex integrals.
\newblock In {\em Annales de l'Institut Henri Poincar{\'e} C, Analyse Non Lin{\'e}aire}, volume~3, pages 185--208. Elsevier, 1986.

\bibitem{giaquinta1986remarks}
M.~Giaquinta, G.~Modica, et~al.
\newblock Remarks on the regularity of minimizers of certain degenerate functionals.
\newblock {\em Manuscripta Mathematica}, 57:55--99, 1986.

\bibitem{giusti2003direct}
E.~Giusti.
\newblock {\em Direct methods in the calculus of variations}.
\newblock World Scientific, 2003.

\bibitem{grisvard2011elliptic}
P.~Grisvard.
\newblock {\em Elliptic problems in nonsmooth domains}, volume~24.
\newblock Pitman Advanced Publishing Program, Boston, MA, 1985.

\bibitem{hamburger1992regularity}
C.~Hamburger.
\newblock Regularity of differential forms minimizing degenerate elliptic functionals.
\newblock {\em Reine Angew. Math.}, 431:7--64, 1992.

\bibitem{jiangsheng1998regularity}
Y.~Jiangsheng.
\newblock Regularity of solutions of certain parabolic system with nonstandard growth condition.
\newblock {\em Acta Mathematica Sinica}, 14(2):145--160, 1998.

\bibitem{lieberman1988boundary}
G.M. Lieberman.
\newblock Boundary regularity for solutions of degenerate elliptic equations.
\newblock {\em Nonlinear Analysis: Theory, Methods \& Applications}, 12(11):1203--1219, 1988.

\bibitem{lieberman1990boundary}
G.M. Lieberman.
\newblock Boundary regularity for solutions of degenerate parabolic equations.
\newblock {\em Nonlinear Analysis: Theory, Methods \& Applications}, 14(6):501--524, 1990.

\bibitem{lions1969quelques}
J.L. Lions.
\newblock {Q}uelques m{\'e}thodes de r{\'e}solution des probl{\`e}mes aux limites non-lin{\'e}aires.
\newblock {\em Dunod}, Paris, 1969.

\bibitem{marcellinieins}
P.~Marcellini.
\newblock Regularity of minimizers of integrals of the calculus of variations with non standard growth conditions.
\newblock {\em Archive for Rational Mechanics and Analysis}, 105:267--284, 1989.

\bibitem{marcellinizwei}
P.~Marcellini.
\newblock Regularity and existence of solutions of elliptic equations with $p, q$-growth conditions.
\newblock {\em Journal of Differential Equations}, 90(1):1--30, 1991.

\bibitem{marcellinidrei}
P.~Marcellini.
\newblock Regularity for some scalar variational problems under general growth conditions.
\newblock {\em Journal of Optimization Theory and Applications}, 90:161--181, 1996.

\bibitem{marcellinivier}
P.~Marcellini.
\newblock Growth conditions and regularity for weak solutions to nonlinear elliptic pdes.
\newblock {\em Journal of Mathematical Analysis and Applications}, 501(1):124408, 2021.

\bibitem{marcellini2006nonlinear}
P.~Marcellini and G.~Papi.
\newblock Nonlinear elliptic systems with general growth.
\newblock {\em Journal of Differential Equations}, 221(2):412--443, 2006.

\bibitem{mingionepartial}
G.~Mingione.
\newblock Regularity of minima: an invitation to the dark side of the calculus of variations.
\newblock {\em Applications of Mathematics}, 51(4):355--426, 2006.

\bibitem{schmidt2008regularity}
T.~Schmidt.
\newblock Regularity theorems for degenerate quasiconvex energies with $(p,q)$-growth.
\newblock {\em Advances in Calculus of Variations}, 1:241--270, 2008.

\bibitem{schmidt2009regularity}
T.~Schmidt.
\newblock Regularity of relaxed minimizers of quasiconvex variational integrals with $(p, q)$-growth.
\newblock {\em Archive for Rational Mechanics and Analysis}, 193(2):311--337, 2009.

\bibitem{singer2015parabolic}
T.~Singer.
\newblock Parabolic equations with $p, q$-growth: the subquadratic case.
\newblock {\em Quarterly Journal of Mathematics}, 66(2):707--742, 2015.

\bibitem{tolksdorf1983everywhere}
P.~Tolksdorf.
\newblock Everywhere-regularity for some quasilinear systems with a lack of ellipticity.
\newblock {\em Annali di Matematica Pura ed Applicata}, 134(1):241--266, 1983.

\bibitem{tolksdorf1984regularity}
P.~Tolksdorf.
\newblock Regularity for a more general class of quasilinear elliptic equations.
\newblock {\em Journal of Differential Equations}, 51(1):126--150, 1984.

\bibitem{uhlenbeck1977regularity}
K.~Uhlenbeck.
\newblock Regularity for a class of non-linear elliptic systems.
\newblock {\em Acta Mathematica}, 138:219--240, 1977.

\bibitem{ural1968degenerate}
N.N. Ural’ceva.
\newblock Degenerate quasilinear elliptic systems.
\newblock {\em Zap. Na. Sem. Leningrad. Otdel. Mat. Inst. Steklov.(LOMI)}, 7:184--222, 1968.

\end{thebibliography}


\end{document}